\documentclass[12pt]{amsart}
\usepackage{amsmath,amsthm,amsfonts,amssymb}
\usepackage[
margin=2.5cm]{geometry}
\usepackage[usenames,dvipsnames]{color}
\usepackage{mathrsfs}
\usepackage[colorlinks,allcolors=blue]{hyperref}
\usepackage{tikz,tikz-qtree}
\usetikzlibrary{matrix,arrows,trees,shapes}
\usepackage{xcolor}
\usepackage{enumitem}
\setlist[enumerate]{%
  leftmargin=*,%
  label=(\roman*)} 
\setlist[itemize]{leftmargin=*}
\usepackage[mathscr]{eucal}
\usepackage{multirow, array}

\usepackage[yyyymmdd]{datetime}

\usepackage{fancyhdr}
\newtheorem{theorem}{Theorem}[section]
\newtheorem{lemma}[theorem]{Lemma}
\newtheorem{proposition}[theorem]{Proposition}
\newtheorem{corollary}[theorem]{Corollary}

\theoremstyle{definition}%
\newtheorem{example}[theorem]{Example}
\newtheorem{remark}[theorem]{Remark}
\newtheorem{algorithm}[theorem]{Algorithm}

\DeclareMathOperator{\hilb}{Hilb}
\DeclareMathOperator{\PP}{\mathbb{P}}
\DeclareMathOperator{\Macaulaytree}{\mathscr{M}}
\DeclareMathOperator{\hilbtree}{\mathscr{H}}

\newcommand{\hf}{\mathsf{h}}
\newcommand{\hp}{\mathsf{p}}
\newcommand{\hq}{\mathsf{q}}
\newcommand{\lift}{\Phi}
\newcommand{\plus}{A}
\newcommand{\hs}{\mathsf{H}}
\newcommand{\hk}{\mathsf{K}}

\newcommand{\kk}{\mathbb{K}}
\newcommand{\NN}{\mathbb{N}}
\newcommand{\QQ}{\mathbb{Q}}
\newcommand{\ZZ}{\mathbb{Z}}
\newcommand{\RR}{\mathbb{R}}

\DeclareMathOperator{\lexg}{\succ}

\DeclareMathOperator{\prob}{Pr}
\DeclareMathOperator{\pr}{pr}

\newcommand{\relphantom}[1]{\mathrel{\phantom{ #1 }}}

\begin{document}

\begin{abstract}
  We investigate the geography of Hilbert schemes parametrizing closed
  subschemes of projective space with specified Hilbert polynomials.
  We classify Hilbert schemes with unique Borel-fixed points via
  combinatorial expressions for their Hilbert polynomials.  These
  expressions naturally lead to an arrangement of nonempty Hilbert
  schemes as the vertices of an infinite full binary tree.  Here we
  discover regularities in the geometry of Hilbert schemes.
  Specifically, under natural probability distributions on the tree,
  we prove that Hilbert schemes are irreducible and nonsingular with
  probability greater than $0.5$.
\end{abstract}

\title{The Ubiquity of Smooth Hilbert Schemes}
\author{Andrew~P.~Staal}%
\email{a.staal@queensu.ca}
\address{ Department of Pure Mathematics \\
  University of Waterloo \\
  200 University Avenue West \\
  Waterloo, Ontario, Canada, N2L 3G1 }%
\date{}
\maketitle

\section{Introduction}

Hilbert schemes parametrizing closed subschemes with a fixed Hilbert
polynomial in projective space are fundamental moduli spaces.
However, with the exception of Hilbert schemes parametrizing
hypersurfaces \cite[Example~2.3]{Arbarello--Cornalba--Griffiths--2011}
and points in the plane \cite{Fogarty--1968}, the geometric features
of typical Hilbert schemes are still poorly understood.  Techniques
for producing pathological Hilbert schemes are known, generating
Hilbert schemes with many irreducible components
\cite{Iarrobino--1972, Fantechi--Pardini--1996}, with generically
nonreduced components \cite{Mumford--1962}, and with arbitrary
singularity types \cite{Vakil--2006}.  This raises the questions: What
should we expect from a random Hilbert scheme?  Can we understand the
geography of Hilbert schemes?  Our answer is that the set of nonempty
Hilbert schemes forms a collection of trees and a discrete probability
space, and that irreducible, nonsingular Hilbert schemes are
unexpectedly common.

Let $\hilb^{\hp}(\PP^n)$ be the Hilbert scheme parametrizing closed
subschemes of $\PP^n_{\kk}$ with Hilbert polynomial $\hp$, where $\kk$
is a field.  Macaulay classified Hilbert polynomials of homogeneous
ideals in \cite{Macaulay--1927}.  Any such admissible Hilbert
polynomial $\hp(t)$ has a unique combinatorial expression of the form
$\sum_{j=1}^r \binom{t+b_j-j+1}{b_j}$, for integers $b_1 \ge b_2 \ge
\dotsb \ge b_r \ge 0$.  Our main result is the following theorem.

\begin{theorem}
  \label{thm:SSSunique2intro}
  The lexicographic ideal is the unique saturated strongly stable
  ideal of codimension $c$ with Hilbert polynomial $\hp$ if and only
  if at least one of the following holds:
  \begin{enumerate}
  \item $b_r > 0$,
  \item $c \ge 2$ and $r \le 2$,
  \item $c = 1$ and $b_1 = b_r$, or
  \item $c = 1$ and $r - s \le 2$, where $b_1 = b_2 = \dotsb = b_{s} >
    b_{s+1} \ge \dotsb \ge b_r$.
  \end{enumerate}
  If $\kk$ is algebraically closed, then the lexicographic ideal is
  the unique saturated Borel-fixed ideal of codimension $c$ with
  Hilbert polynomial $\hp$ if and only if at least one of (i)--(iv)
  holds.
\end{theorem}

Strongly stable ideals, including lexicographic ideals, are
Borel-fixed.  In characteristic $0$, the converse also holds.  Many
fundamental properties of Hilbert schemes have been understood through
these ideals.  Hartshorne and later Peeva--Stillman found rational
curves linking Borel-fixed points, proving connectedness
\cite{Hartshorne--1966, Peeva--Stillman--2005}.  Bayer used them to
derive defining equations and proposed studying their tangent cones
\cite{Bayer--1982}.  Reeves, followed by Pardue, studied their
combinatorial properties to give bounds for radii of Hilbert schemes
\cite{Reeves--1995, Pardue--1994} and Reeves--Stillman proved that
lexicographic points are nonsingular \cite{Reeves--Stillman--1997}.
Gotzmann also discovered some irreducible Hilbert schemes in
\cite{Gotzmann--1989}.  Theorem~\ref{thm:SSSunique2intro} advances
this line of inquiry, identifies a large collection of well-behaved
Hilbert schemes that includes Gotzmann's examples, and improves our
understanding of the geography of Hilbert schemes.

To make a quantitative statement about all Hilbert schemes, we
interpret Macaulay's classification as follows: First, identify any
admissible Hilbert polynomial $\hp$ with its sequence $b = (b_1, b_2,
\ldots, b_r)$.  These sequences are generated by two operations,
namely ``integrating'' $\hp$ to $\lift (\hp) := b + (1,1, \ldots, 1)$
and ``adding one'' to $\hp$ to get $\plus (\hp) := 1 + \hp = (b, 0)$.
The set of all such sequences forms an infinite full binary tree.
There is an associated tree, which we denote $\hilbtree_c$, whose
vertices are the Hilbert schemes $\hilb^{\hp}(\PP^n)$ parametrizing
codimension $c = n - \deg \hp$ subschemes, for each positive $c \in
\ZZ$.  Geometrically, $\lift$ corresponds to coning over parametrized
schemes and $\plus$ to adding a point.  We endow $\hilbtree_c$ with a
natural probability distribution, in which the vertices at a fixed
height are equally likely.  This leads to our second main result.

\begin{theorem}
  \label{thm:Hilbertprobintro}
  Let $\kk$ be algebraically closed or have characteristic $0$.  The
  probability that a random Hilbert scheme is irreducible and
  nonsingular is greater than $0.5$.
\end{theorem}

This theorem counterintuitively suggests that the geometry of the
majority of Hilbert schemes is understandable.  To prove
Theorems~\ref{thm:SSSunique2intro} and \ref{thm:Hilbertprobintro}, we
study the algorithm generating saturated strongly stable ideals first
described by Reeves \cite{Reeves--1992} and later generalized in
\cite{Moore--2012, Cioffi--Lella--Marinari--Roggero--2011}.  We obtain
precise information about Hilbert series and $K$-polynomials of
saturated strongly stable ideals.  The primary technical result we
need is the following.

\begin{theorem}
  \label{thm:Kpolydegreeintro}
  Let $I \subset \kk[x_0, x_1, \dotsc , x_n]$ be a saturated strongly
  stable ideal with Hilbert polynomial $\hp$, let $L^{\hp}_n$ be the
  corresponding lexicographic ideal in $\kk[x_0, x_1, \dotsc, x_n]$,
  and let $\hk_I$ be the numerator of the Hilbert series of $I$.  If
  $I \ne L^{\hp}_n$, then we have $\deg \hk_{I} < \deg
  \hk_{L^{\hp}_n}$.
\end{theorem}

The structure of the paper is as follows.  In
Section~\ref{ch:Macaulaytree}, we introduce two binary relations on
the set of admissible Hilbert polynomials and show that they generate
all such polynomials.  The set of lexicographic ideals is then
partitioned by codimension into infinitely many binary trees in
Section~\ref{ch:lexforest}.  Geometrically, these are trees of Hilbert
schemes, as every Hilbert scheme contains a unique lexicographic
ideal.  To identify well-behaved Hilbert schemes, we review saturated
strongly stable ideals in Section~\ref{ch:stronglystable} and we
examine their $K$-polynomials in Section~\ref{ch:Kpolynomials}.  The
main results are in Section~\ref{ch:Hilbertirred}.

\subsection*{Conventions.}

Throughout, $\kk$ is a field, $\NN$ is the nonnegative integers, and
$\kk[x_0, x_1, \dotsc, x_n]$ is the standard $\ZZ$-graded polynomial
ring.  The Hilbert function, polynomial, series, and
$K$\nobreakdash-polynomial of the quotient $\kk[x_0, x_1, \dotsc, x_n]
/ I$ by a homogeneous ideal $I$ are denoted $\hf_I, \hp_I, \hs_I$, and
$\hk_I$, respectively.

\subsection*{Acknowledgments}

We especially thank Gregory~G.~Smith for his guidance in this
research.  We thank Mike~Roth, Ivan~Dimitrov, Tony~Geramita,
Chris~Dionne, Ilia~Smirnov, Nathan~Grieve, Andrew~Fiori, Simon~Rose,
and Alex~Duncan for many discussions.  We also thank the anonymous
referee for helpful remarks that improved the paper.  This research
was supported by an E.G.~Bauman Fellowship in 2011-12, by Ontario
Graduate Scholarships in 2012-15, and by Gregory~G.~Smith's NSERC
Discovery Grant in 2015-16.

\section{Binary Trees and Hilbert Schemes}

In \ref{ch:Macaulaytree} we observe that a tree structure exists on
the set of numerical polynomials determining nonempty Hilbert schemes.
Macaulay's pioneering work \cite{Macaulay--1927} classifies these
polynomials and two mappings turn this set into an infinite binary
tree.  In \ref{ch:lexforest} we find related binary trees in the sets
of lexicographic ideals and Hilbert schemes.

\subsection{The Macaulay Tree}
\label{ch:Macaulaytree}

Let $\kk$ be a field and let $\kk[x_0, x_1, \dotsc ,x_n]$ denote the
homogeneous (standard $\ZZ$-graded) coordinate ring of $n$-dimensional
projective space $\PP^n_{\kk}$.  Let $M$ be a finitely generated
graded $\kk[x_0, x_1, \dotsc, x_n]$\nobreakdash-module.  The
\emph{\bfseries Hilbert function} $\hf_M \colon \ZZ \to \ZZ$ of $M$ is
defined by $\hf_M(i) := \dim_{\kk} (M_i)$ for all $i \in \ZZ$.  Every
such $M$ has a \emph{\bfseries Hilbert polynomial} $\hp_M$, that is, a
polynomial $\hp_M(t) \in \QQ[t]$ such that $\hf_M(i) = \hp_M(i)$ for
$i \gg 0$; see \cite[Theorem~4.1.3]{Bruns--Herzog--1993}.  For a
homogeneous ideal $I \subset \kk[x_0, x_1, \dotsc, x_n]$, let $\hf_I$
and $\hp_I$ denote the Hilbert function and Hilbert polynomial of the
quotient module $\kk[x_0, x_1, \dotsc, x_n]/I$, respectively.

We begin with a basic example and make a notational convention, for
later use.

\begin{example}
  \label{eg:myfirstpoly} 
  Fix a nonnegative integer $n$.  By the classic stars-and-bars
  argument \cite[Section~1.2]{Stanley--2012}, we have $\hf_{S}(i) =
  \binom{i + n}{n}$ for all $i \in \ZZ$, where $S = \kk[x_0, x_1,
    \dotsc, x_n]$.  The equality $\hf_{S}(i) = \hp_{S}(i)$ is only
  valid for $i \ge -n$, because the polynomial $\hp_{S}(t) = \binom{t
    + n}{n} \in \QQ[t]$ only has roots $-n$, $-n+1, \dotsc, -1$,
  whereas $\hf_{S}(i) = 0$ for all $i < 0$.
\end{example}

\begin{remark}
  \label{rmk:binconv}
  For integers $j,k$ we set $\binom{j}{k} = \frac{j!}{k!(j-k)!}$ if $j
  \ge k \ge 0$ and $\binom{j}{k} = 0$ otherwise.  For a variable $t$
  and $a, b \in \ZZ$, we define $\binom{t + a}{b} = \frac{(t + a)(t +
    a - 1) \dotsb (t + a - b + 1)}{b!} \in \QQ[t]$ if $b \ge 0$, and
  $\binom{t + a}{b} = 0$ otherwise.  When $b \ge 0$, the polynomial
  $\binom{t + a}{b}$ has degree $b$ with zeros $-a, -a+1, \dotsc,
  -a+b-1$.  Importantly, we have $\binom{t + a}{b} \rvert_{t = j} \ne
  \binom{j + a}{b} = 0$ when $j < -a$.  Interestingly,
  \cite[p.533]{Macaulay--1927} uses distinct notation for polynomial
  and integer binomial coefficients.
\end{remark}

A polynomial is an \emph{\bfseries admissible Hilbert polynomial} if
it is the Hilbert polynomial of some closed subscheme in some $\PP^n$.
Admissible Hilbert polynomials correspond to nonempty Hilbert schemes.
We use the well-known classification first discovered by Macaulay.

\begin{proposition} 
  \label{prop:expressions}
  The following conditions are equivalent:
  \begin{enumerate}
  \item The polynomial $\hp(t) \in \QQ[t]$ is a nonzero admissible
    Hilbert polynomial.
    
  \item There exist integers $e_0 \ge e_1 \ge \dotsb \ge e_d > 0$ such
    that $\hp(t) = \sum_{i=0}^{d} \binom{t + i}{i + 1} - \binom{t + i
      - e_i}{i + 1}$.

  \item There exist integers $b_1 \ge b_2 \ge \dotsb \ge b_r \ge 0$
    such that $\hp(t) = \sum_{j=1}^r \binom{t + b_j - j+1}{b_j}$.
  \end{enumerate}
  Moreover, these correspondences are bijective.
\end{proposition}

\begin{proof} $\;$
  \begin{enumerate}[leftmargin=2cm]
  \item[(i) $\Leftrightarrow$ (ii)] This is proved in
    \cite{Macaulay--1927}; see the formula for ``$\chi(\ell)$'' at the
    bottom of p.536.  For a geometric account, see \cite[Corollary~3.3
      and Corollary~5.7]{Hartshorne--1966}.

  \item[(i) $\Leftrightarrow$ (iii)] This follows from
    \cite[Erinnerung~2.4]{Gotzmann--1978}; see also
    \cite[Exercise~4.2.17]{Bruns--Herzog--1993}.
  \end{enumerate}
  Uniqueness of the sequences of integers also follows.
\end{proof}

For simplicity, we always work with nonzero admissible Hilbert
polynomials.  Let the \emph{\bfseries Macaulay--Hartshorne expression}
of an admissible Hilbert polynomial $\hp$ be its expression $\hp(t) =
\sum_{i=0}^{d} \binom{t + i}{i + 1} - \binom{t + i - e_i}{i + 1}$, for
$e_0 \ge e_1 \ge \dotsb \ge e_d > 0$, and the \emph{\bfseries Gotzmann
  expression} of $\hp$ be its expression $\hp(t) = \sum_{j=1}^{r}
\binom{t + b_j - j+1}{b_j}$, for $b_1 \ge b_2 \ge \dotsb \ge b_r \ge
0$.  From these, we find the degree $d = b_1$, the leading coefficient
$e_d / d!$, and the \emph{\bfseries Gotzmann number} $r$ of $\hp$,
which bounds the Castelnuovo--Mumford regularity of saturated ideals
with Hilbert polynomial $\hp$.  In particular, such ideals are
generated in degree $r$; see
\cite[p.~300-301]{Iarrobino--Kanev--1999}.

Macaulay--Hartshorne and Gotzmann expressions are conjugate.  Recall
that the \emph{\bfseries conjugate} partition to a partition
$\lambda = \left( \lambda_1, \lambda_2, \dotsc, \lambda_k \right)$ of
an integer $\ell = \sum_{i=1}^k \lambda_i$ is the partition of $\ell$
obtained from the Ferrers diagram of $\lambda$ by interchanging rows
and columns, having $\lambda_i - \lambda_{i+1}$ parts equal to $i$;
see \cite[Section~1.8]{Stanley--2012}.

\begin{lemma}
  \label{lem:dictionary}
  If $\hp(t) \in \QQ[t]$ is an admissible Hilbert polynomial with
  Macaulay--Hartshorne expression
  $\sum_{i=0}^d \binom{t+i}{i+1} - \binom{t+i -e_i}{i+1}$ for
  $e_0 \ge e_1 \ge \dotsb \ge e_d > 0$ and Gotzmann expression
  $\sum_{j=1}^r \binom{t + b_j - j+1}{b_j}$ for
  $b_1 \ge b_2 \ge \dotsb \ge b_r \ge 0$, then $r = e_0$ and the
  nonnegative partition $( b_1, b_2, \dotsc, b_r )$ is conjugate to
  the partition $( e_1, e_2, \dotsc, e_d )$.
\end{lemma}

\begin{proof}
  The key step is to rewrite $\hp$ as $\sum_{i=0}^{d} \binom{t + i}{i
    + 1} - \binom{t + i - e_d}{i + 1} + \sum_{i=0}^{d-1} \binom{t + i
    - e_d}{i + 1} - \binom{t + i - e_i}{i + 1}$ and to prove
  $\sum_{i=0}^{d} \binom{t + i}{i + 1} - \binom{t + i - e_d}{i + 1} =
  \sum_{j=1}^{e_d} \binom{t+d-j+1}{d}$, by induction on $d$.  This
  gives the expression
  \[
  \hp(t) = \sum_{j=1}^{e_d} \binom{t + d - j+1}{d} + \left[
    \sum_{i=0}^{d-1} \binom{s + i}{i + 1} - \binom{s + i - e_i+e_d}{i
      + 1} \right]_{s = t - e_d}
  \]
  and one can iterate on the second part.  So $r = e_0$ and the
  partition $(b_1, b_2, \ldots, b_r)$ has $e_i - e_{i+1}$ parts equal
  to $i$, for all $0 \le i \le d$.  The equalities $\sum_{j=1}^r b_j =
  \sum_{i=0}^d (e_i - e_{i+1})i = \sum_{i=1}^d e_i$ then show that $(
  b_1, b_2, \dotsc, b_r )$ is conjugate to $( e_1, e_2, \dotsc, e_d
  )$.
\end{proof}

For $\hp(t)$ as in Lemma~\ref{lem:dictionary}, it is convenient to
refer to $(e_0, e_1, \dotsc, e_d)$ as its \emph{\bfseries
  Macaulay--Hartshorne partition} and $(b_1, b_2, \dotsc, b_r)$ as its
(nonnegative) \emph{\bfseries Gotzmann partition}.

We now describe two fundamental binary relations on admissible Hilbert
polynomials.  The first takes the polynomial $\hp$ with partitions $e
= (e_0, e_1, \dotsc, e_d)$ and $b = (b_1, b_2, \dotsc, b_r)$ to the
polynomial $\lift (\hp)$ with partitions $(e_0, e)$ and $b + (1, 1,
\ldots, 1)$ (add one to each entry).  The second takes $\hp$ to $\plus
(\hp) = 1 + \hp$, with partitions $e + (1, 0, \ldots, 0)$ and $(b,
0)$.  Both $\lift (\hp)$ and $\plus (\hp)$ are admissible by
Proposition~\ref{prop:expressions}.

The \emph{\bfseries backwards difference operator} $\nabla$ maps any
$\hq \in \QQ[t]$ to $\hq(t) - \hq(t-1)$.  Backwards differences are
discrete derivatives---in Lemma~\ref{lem:liftproperties}, (ii) says
that $\lift$ is the indefinite integral and (iii) is a well-known
discrete analogue of the Fundamental Theorem of Calculus.

\begin{lemma} 
  \label{lem:liftproperties} 
  If $\hp(t)$ is an admissible Hilbert polynomial with
  Macaulay--Hartshorne partition $(e_0, e_1, \dotsc, e_d)$ and
  Gotzmann partition $( b_1, b_2, \dotsc, b_r )$, then the following
  hold:
  \begin{enumerate}
  \item $\left[ \nabla (\hp) \right](t) = \sum_{j=1}^{r} \binom{t +
    b_j-1 - j+1}{b_j-1} = \sum_{i=0}^{d-1} \binom{t + i}{i + 1} -
    \binom{t + i - e_{i+1}}{i + 1}$;

  \item $\nabla \plus^a \lift (\hp) = \hp$, for all $a \in \NN$; and

  \item if $\deg \hp > 0$ and $k \in \{ 1, 2, \dotsc, r \}$ is the
    largest index such that $b_k \neq 0$, then we have
    $\hp - \lift \nabla (\hp) = r - k$, but if $\deg \hp = 0$, then
    $\nabla (\hp) = 0$.
\end{enumerate}
\end{lemma}

\begin{proof}
These follow by linearity of $\nabla$ and the binomial addition
formula.
\end{proof}

We now observe that the set of admissible Hilbert polynomials forms a
tree.

\begin{proposition}
  \label{prop:Macaulaytree}
  The tree with vertices corresponding to admissible Hilbert
  polynomials and edges corresponding to pairs of the form $\bigl(
  \hp, \plus (\hp) \bigr)$ and $\bigl( \hp, \lift (\hp) \bigr)$, for
  all admissible Hilbert polynomials $\hp$, forms an infinite full
  binary tree.  The root of the tree corresponds to $1$.
\end{proposition}

We call this the \emph{\bfseries Macaulay tree} $\Macaulaytree$.  It
has $2^j$ vertices at height $j$, for all $j \in \NN$.

\begin{proof}
  By induction on $r$, $\hp(t) = \sum_{j=1}^r \binom{t + b_j -
    j+1}{b_j} = \lift^{b_r} \plus \lift^{b_{r-1} - b_r} \plus \dotsb
  \plus \lift^{b_2 - b_3} \plus \lift^{b_1 - b_2} (1)$ holds.
\end{proof}

A portion of $\Macaulaytree$ is displayed in
Figure~\ref{fig:Macaulaytree}, in terms of Gotzmann expressions.

\begin{figure}[!ht] 
  \centering
  \resizebox{\linewidth}{!}{
    \begin{tikzpicture}[ 
      grow = right, 
      level distance = 250pt, 
      sibling distance = 25pt ]
      \tikzset{ 
        edge from parent/.style={ 
          draw, 
          very thick, 
        }, 
        every tree node/.style={ 
          draw,
          very thick,
          shape = ellipse,
          fill = yellow!25,
          font = \LARGE,
          text width=150pt, 
          align = center } } 
      \Tree 
      [.$\binom{t}{0}$ 
        \edge node[auto=right] {\huge$\lift$};[.$\binom{t + 1}{1}$ 
          [.$\binom{t + 2}{2}$
            [.$\binom{t + 3}{3}$
              [.$\binom{t + 4}{4}$
                \edge[dashed]; \node[text width=0pt, draw=none, fill=none]{}; 
                \edge[dashed]; \node[text width=0pt, draw=none, fill=none]{}; ]
              [.{$\binom{t + 3}{3} + \binom{t - 1}{0}$}
                \edge[dashed]; \node[text width=0pt, draw=none, fill=none]{}; 
                \edge[dashed]; \node[text width=0pt, draw=none, fill=none]{}; ] ] 
            [.{$\binom{t + 2}{2} + \binom{t - 1}{0}$}
              [.{$\binom{t + 3}{3} + \binom{t}{1}$}
                \edge[dashed]; \node[text width=0pt, draw=none, fill=none]{}; 
                \edge[dashed]; \node[text width=0pt, draw=none, fill=none]{}; ] 
              [.{$\binom{t + 2}{2} + \binom{t - 1}{0} + \binom{t - 2}{0}$}
                \edge[dashed]; \node[text width=0pt, draw=none, fill=none]{}; 
                \edge[dashed]; \node[text width=0pt, draw=none, fill=none]{}; ] ] ] 
          [.{$\binom{t + 1}{1} + \binom{t - 1}{0}$}
            [.{$\binom{t + 2}{2} + \binom{t}{1}$}
              [.{$\binom{t + 3}{3} + \binom{t + 1}{2}$}
                \edge[dashed]; \node[text width=0pt, draw=none, fill=none]{}; 
                \edge[dashed]; \node[text width=0pt, draw=none, fill=none]{}; ] 
              [.{$\binom{t + 2}{2} + \binom{t}{1} + \binom{t - 2}{0}$}
                \edge[dashed]; \node[text width=0pt, draw=none, fill=none]{}; 
                \edge[dashed]; \node[text width=0pt, draw=none, fill=none]{}; ] ]
            [.{$\binom{t + 1}{1} + \binom{t - 1}{0} + \binom{t - 2}{0}$}
              [.{$\binom{t + 2}{2} + \binom{t}{1} + \binom{t - 1}{1}$} 
                \edge[dashed]; \node[text width=0pt, draw=none, fill=none]{}; 
                \edge[dashed]; \node[text width=0pt, draw=none, fill=none]{}; ] 
              [.{$\binom{t + 1}{1} + \binom{t - 1}{0} + \binom{t -
                    2}{0} + \binom{t - 3}{0}$} 
                \edge[dashed]; \node[text width=0pt, draw=none, fill=none]{}; 
                \edge[dashed]; \node[text width=0pt, draw=none, fill=none]{}; ] ] ] ] 
        \edge node[auto=left] {\huge$\plus$};[.{$\binom{t}{0} + \binom{t - 1}{0}$} 
          [.{$\binom{t + 1}{1} + \binom{t}{1}$}
            [.{$\binom{t + 2}{2} + \binom{t + 1}{2}$} 
              [.{$\binom{t + 3}{3} + \binom{t + 2}{3}$}
                \edge[dashed]; \node[text width=0pt, draw=none, fill=none]{}; 
                \edge[dashed]; \node[text width=0pt, draw=none, fill=none]{}; ] 
              [.{$\binom{t + 2}{2} + \binom{t + 1}{2} + \binom{t - 2}{0}$}
                \edge[dashed]; \node[text width=0pt, draw=none, fill=none]{}; 
                \edge[dashed]; \node[text width=0pt, draw=none, fill=none]{}; ] ] 
            [.{$\binom{t + 1}{1} + \binom{t}{1} + \binom{t - 2}{0}$}
              [.{$\binom{t + 2}{2} + \binom{t + 1}{2} + \binom{t - 1}{1}$} 
                \edge[dashed]; \node[text width=0pt, draw=none, fill=none]{}; 
                \edge[dashed]; \node[text width=0pt, draw=none, fill=none]{}; ]
              [.{$\binom{t + 1}{1} + \binom{t}{1} + \binom{t - 2}{0} +
                  \binom{t - 3}{0}$} 
                \edge[dashed]; \node[text width=0pt, draw=none, fill=none]{}; 
                \edge[dashed]; \node[text width=0pt, draw=none, fill=none]{}; ] ] ] 
          [.{$\binom{t}{0} + \binom{t - 1}{0} + \binom{t - 2}{0}$} 
            [.{$\binom{t + 1}{1} + \binom{t}{1} + \binom{t - 1}{1}$} 
              [.{$\binom{t + 2}{2} + \binom{t + 1}{2} + \binom{t}{2}$} 
                \edge[dashed]; \node[text width=0pt, draw=none, fill=none]{}; 
                \edge[dashed]; \node[text width=0pt, draw=none, fill=none]{}; ] 
              [.{$\binom{t + 1}{1} + \binom{t}{1} + \binom{t - 1}{1} +
                  \binom{t - 3}{0}$} 
                \edge[dashed]; \node[text width=0pt, draw=none, fill=none]{}; 
                \edge[dashed]; \node[text width=0pt, draw=none, fill=none]{}; ] ] 
            [.{$\binom{t}{0} + \binom{t - 1}{0} + \binom{t - 2}{0} + \binom{t - 3}{0}$} 
              [.{$\binom{t + 1}{1} + \binom{t}{1} + \binom{t - 1}{1} +
                  \binom{t - 2}{1}$} 
                \edge[dashed]; \node[text width=0pt, draw=none, fill=none]{}; 
                \edge[dashed]; \node[text width=0pt, draw=none, fill=none]{}; ] 
              [.{$\binom{t}{0} + \binom{t - 1}{0} + \binom{t - 2}{0} +
                  \binom{t - 3}{0} + \binom{t - 4}{0}$} 
                \edge[dashed]; \node[text width=0pt, draw=none, fill=none]{}; 
                \edge[dashed]; \node[text width=0pt, draw=none, fill=none]{}; ] ] ] ] ]
    \end{tikzpicture}
  }
  \caption{The Macaulay tree $\Macaulaytree$ to height $4$ with
    Gotzmann expressions \label{fig:Macaulaytree}}
\end{figure}
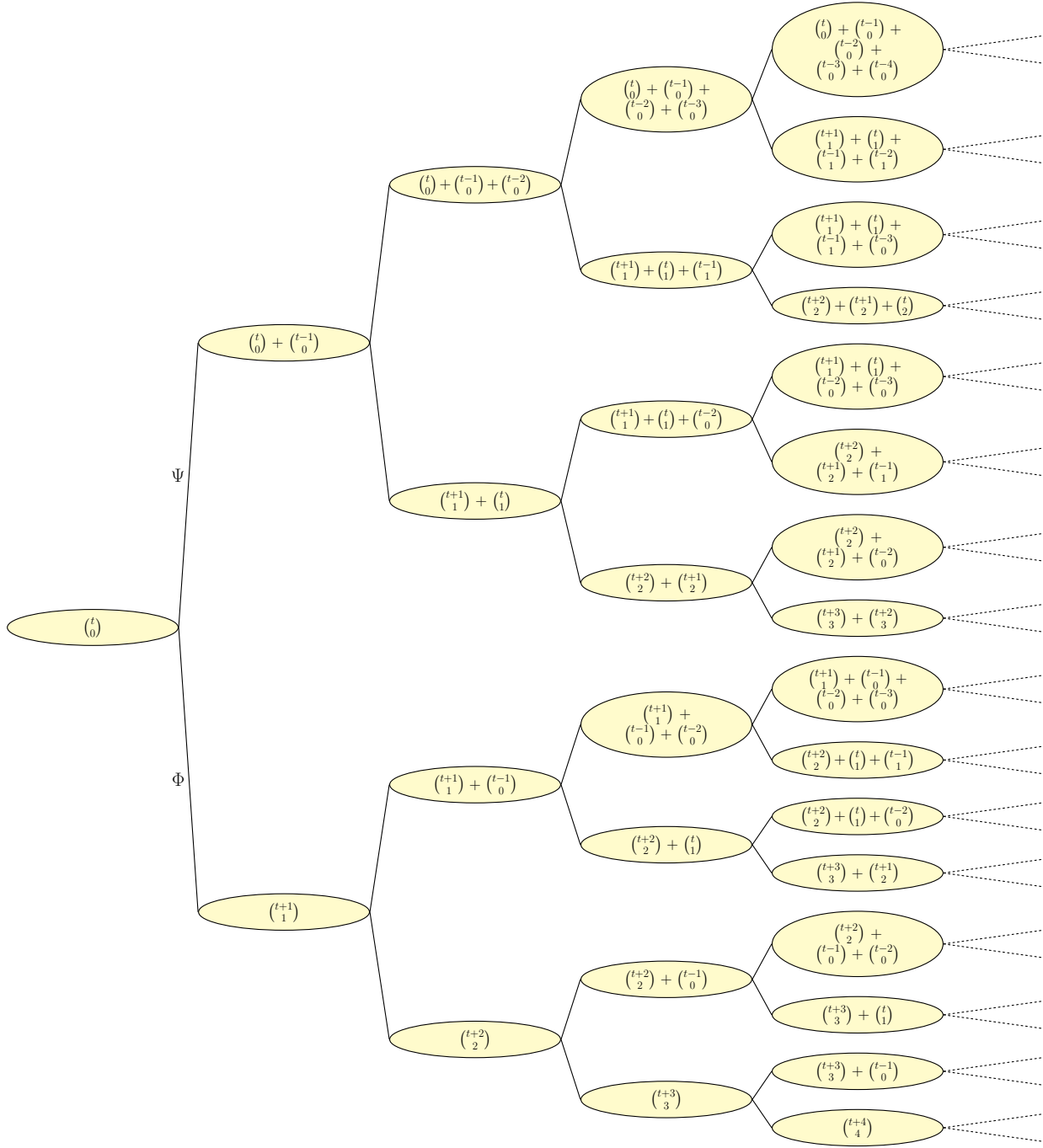

\begin{remark}
  \label{rmk:pathexpressions} 
  The path from the root $1$ of the tree $\Macaulaytree$ to $\hp(t) =
  \sum_{i=0}^d \binom{t + i}{i+1} - \binom{t + i - e_i}{i+1}$ can also
  be expressed as $\hp = \plus^{e_0 - e_1} \lift \plus^{e_1 - e_2}
  \lift \dotsb \lift \plus^{e_{d-1} - e_d} \lift \plus^{e_d - 1} (1)$.
\end{remark}

\begin{example}
  The Hilbert polynomial $3t+1$ of the twisted cubic curve $X \subset
  \PP^3$ has partitions $(b_1, b_2, b_3, b_4) = (1,1,1,0)$ and $(e_0,
  e_1) = (4,3)$.  The path in $\Macaulaytree$ from $1$ to $3t+1$ can
  be written as $\lift^0 \plus \lift^{1-0} \plus \lift^{1-1} \plus
  \lift^{1-1} (1) = \plus^{4-3} \lift \plus^{3-1}(1) = \plus \lift
  \plus^{2}(1)$.  This path is shown in Figure~\ref{fig:twistedpath}.
\end{example}

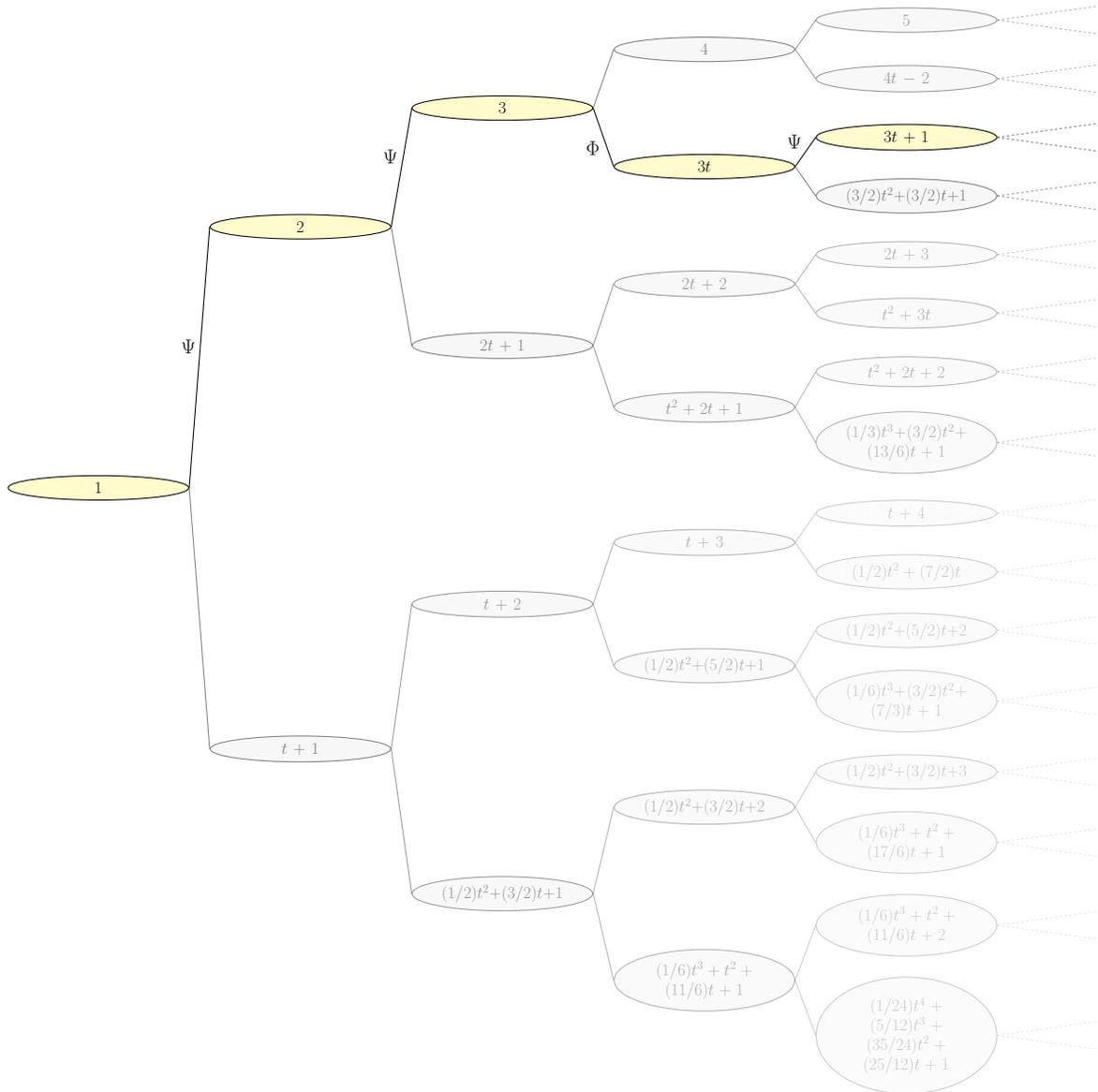
\begin{figure}[!ht] 
  \centering
  \resizebox{.95\linewidth}{!}{
    \begin{tikzpicture}[ 
      grow = right, 
      level distance = 250pt, 
      sibling distance = 25pt ]
      \tikzset{ 
        edge from parent/.style={ 
          draw, 
          very thick, 
          transparent }, 
        every tree node/.style={ 
          draw,
          very thick,
          shape = ellipse,
          fill = lightgray!25,
          font = \LARGE,
          text width=150pt, 
          align = center,
          transparent },
        transparent/.style={
          opacity=0.5,
          text opacity=0.5} } 
      \Tree 
      [.\node[opacity=1, text opacity=1, fill=yellow!25]{$1$}; 
        [.$t+1$ 
          \edge[opacity=0.4];[.\node[opacity=0.4]{$(1/2)t^{2}+(3/2)t+1$}; 
            \edge[opacity=0.3];[.\node[opacity=0.3]{$(1/6)t^{3}+t^{2}+(11/6)t+1$}; 
              \edge[opacity=0.2];[.\node[opacity=0.2]{$(1/24)t^{4}+(5/12)t^{3}+(35/24)t^{2}+(25/12)t+1$}; 
                \edge[opacity=0.1,dashed]; \node[text width=0pt, draw=none, fill=none]{}; 
                \edge[opacity=0.1,dashed]; \node[text width=0pt, draw=none, fill=none]{}; ]
              \edge[opacity=0.2];[.\node[opacity=0.2]{$(1/6)t^{3}+t^{2}+(11/6)t+2$}; 
                \edge[opacity=0.1,dashed]; \node[text width=0pt, draw=none, fill=none]{}; 
                \edge[opacity=0.1,dashed]; \node[text width=0pt, draw=none, fill=none]{}; ] ] 
            \edge[opacity=0.3];[.\node[opacity=0.3]{$(1/2)t^{2}+(3/2)t+2$}; 
              \edge[opacity=0.2];[.\node[opacity=0.2]{$(1/6)t^{3}+t^{2}+(17/6)t+1$}; 
                \edge[opacity=0.1,dashed]; \node[text width=0pt, draw=none, fill=none]{}; 
                \edge[opacity=0.1,dashed]; \node[text width=0pt, draw=none, fill=none]{}; ] 
              \edge[opacity=0.2];[.\node[opacity=0.2]{$(1/2)t^{2}+(3/2)t+3$};
                \edge[opacity=0.1,dashed]; \node[text width=0pt, draw=none, fill=none]{}; 
                \edge[opacity=0.1,dashed]; \node[text width=0pt, draw=none, fill=none]{}; ] ] ] 
          \edge[opacity=0.4];[.\node[opacity=0.4]{$t+2$};
            \edge[opacity=0.3];[.\node[opacity=0.3]{$(1/2)t^{2}+(5/2)t+1$}; 
              \edge[opacity=0.2];[.\node[opacity=0.2]{$(1/6)t^{3}+(3/2)t^{2}+(7/3)t+1$}; 
                \edge[opacity=0.1,dashed]; \node[text width=0pt, draw=none, fill=none]{}; 
                \edge[opacity=0.1,dashed]; \node[text width=0pt, draw=none, fill=none]{}; ] 
              \edge[opacity=0.2];[.\node[opacity=0.2]{$(1/2)t^{2}+(5/2)t+2$};
                \edge[opacity=0.1,dashed]; \node[text width=0pt, draw=none, fill=none]{}; 
                \edge[opacity=0.1,dashed]; \node[text width=0pt, draw=none, fill=none]{}; ] ]
            \edge[opacity=0.3];[.\node[opacity=0.3]{$t+3$}; 
              \edge[opacity=0.2];[.\node[opacity=0.2]{$(1/2)t^{2}+(7/2)t$}; 
                \edge[opacity=0.1,dashed]; \node[text width=0pt, draw=none, fill=none]{}; 
                \edge[opacity=0.1,dashed]; \node[text width=0pt, draw=none, fill=none]{}; ] 
              \edge[opacity=0.2];[.\node[opacity=0.2]{$t+4$}; 
                \edge[opacity=0.1,dashed]; \node[text width=0pt, draw=none, fill=none]{}; 
                \edge[opacity=0.1,dashed]; \node[text width=0pt, draw=none, fill=none]{}; ] ] ] ] 
        \edge[opacity=1] node[text opacity=1, auto=left] {\huge$\plus$};[.\node[opacity=1, text opacity=1, fill=yellow!25]{$2$}; 
          [.$2t+1$ 
            \edge[opacity=0.4];[.\node[opacity=0.4]{$t^{2}+2t+1$}; 
              \edge[opacity=0.3];[.\node[opacity=0.3]{$(1/3)t^{3}+(3/2)t^{2}+(13/6)t+1$}; 
                \edge[opacity=0.2,dashed]; \node[text width=0pt, draw=none, fill=none]{}; 
                \edge[opacity=0.2,dashed]; \node[text width=0pt, draw=none, fill=none]{}; ] 
              \edge[opacity=0.3];[.\node[opacity=0.3]{$t^{2}+2t+2$}; 
                \edge[opacity=0.2,dashed]; \node[text width=0pt, draw=none, fill=none]{}; 
                \edge[opacity=0.2,dashed]; \node[text width=0pt, draw=none, fill=none]{}; ] ] 
            \edge[opacity=0.4];[.\node[opacity=0.4]{$2t+2$}; 
              \edge[opacity=0.3];[.\node[opacity=0.3]{$t^{2}+3t$}; 
                \edge[opacity=0.2,dashed]; \node[text width=0pt, draw=none, fill=none]{}; 
                \edge[opacity=0.2,dashed]; \node[text width=0pt, draw=none, fill=none]{}; ]
              \edge[opacity=0.3];[.\node[opacity=0.3]{$2t+3$}; 
                \edge[opacity=0.2,dashed]; \node[text width=0pt, draw=none, fill=none]{}; 
                \edge[opacity=0.2,dashed]; \node[text width=0pt, draw=none, fill=none]{}; ] ] ] 
          \edge[opacity=1] node[text opacity=1, auto=left] {\huge$\plus$};[.\node[opacity=1, text opacity=1, fill=yellow!25]{$3$}; 
            \edge[opacity=1] node[text opacity=1, auto=right] {\huge$\lift$};[.\node[opacity=1, text opacity=1, fill=yellow!25]{$3t$}; 
              [.$(3/2)t^{2}+(3/2)t+1$ 
                \edge[opacity=0.4,dashed]; \node[text width=0pt, draw=none, fill=none]{}; 
                \edge[opacity=0.4,dashed]; \node[text width=0pt, draw=none, fill=none]{}; ] 
              \edge[opacity=1] node[text opacity=1, auto=left] {\huge$\plus$};[.\node[opacity=1, text opacity=1, fill=yellow!25]{$3t+1$}; 
                \edge[dashed]; \node[text width=0pt, draw=none, fill=none]{}; 
                \edge[dashed]; \node[text width=0pt, draw=none, fill=none]{}; ] ] 
            [.$4$ 
              \edge[opacity=0.4];[.\node[opacity=0.4]{$4t-2$}; 
                \edge[opacity=0.3,dashed]; \node[text width=0pt, draw=none, fill=none]{}; 
                \edge[opacity=0.3,dashed]; \node[text width=0pt, draw=none, fill=none]{}; ] 
              \edge[opacity=0.4];[.\node[opacity=0.4]{$5$}; 
                \edge[opacity=0.3,dashed]; \node[text width=0pt, draw=none, fill=none]{}; 
                \edge[opacity=0.3,dashed]; \node[text width=0pt, draw=none, fill=none]{}; ] ] ] ] ]
    \end{tikzpicture}
  }
  \caption{The path from $1$ to $\hp(t) = 3t+1$ in the Macaulay
    tree \label{fig:twistedpath}}
\end{figure}

\subsection{Lexicographic and Hilbert Trees}
\label{ch:lexforest}

We now connect lexicographic ideals and Hilbert schemes to the
Macaulay tree $\Macaulaytree$.  Specifically, $\Macaulaytree$
reappears infinitely many times in the set of saturated lexicographic
ideals and the set of Hilbert schemes, with one tree for each positive
codimension.  Two mappings on lexicographic ideals analogous to
$\lift$ and $\plus$ are essential.

For any vector $u = (u_0, u_1, \dotsc, u_n) \in \NN^{n+1}$, let $x^u =
x_0^{u_0} x_1^{u_1} \dotsb x_n^{u_n}$.  The \emph{\bfseries
  lexicographic ordering} is the relation $\lexg$ on the monomials in
$\kk[x_0, x_1, \dotsc , x_n]$ defined by $x^u \lexg x^v$ if the first
nonzero coordinate of $u - v \in \ZZ^{n+1}$ is positive, where $u, v
\in \NN^{n+1}$.

\begin{example}
  \label{eg:lexordering}
  We have $x_0 \lexg x_1 \lexg \dotsb \lexg x_n$ in lexicographic
  order on $\kk[x_0, x_1, \dotsc, x_n]$.  If $n \ge 2$, then $x_0
  x_2^2 \lexg x_1^4 \lexg x_1^3$.
\end{example}

Lexicographic ideals are monomial ideals whose homogeneous pieces are
spanned by maximal monomials in lexicographic order.  For a
homogeneous ideal $I$ in $\kk[x_0, x_1, \dotsc, x_n]$, lexicographic
order gives rise to two monomial ideals associated to $I$.  First, the
\emph{\bfseries lexicographic ideal} for the Hilbert function $\hf_I$
in $\kk[x_0, x_1, \dotsc, x_n]$ is the monomial ideal $L^{\hf_I}_n$
whose $i$th graded piece is spanned by the $\dim_{\kk} I_i =
\hf_{\kk[x_0, x_1, \dotsc, x_n]}(i) - \hf_I(i)$ largest monomials in
$\kk[x_0, x_1, \dotsc, x_n]_i$, for all $i \in \ZZ$.  The equality
$\hf_I = \hf_{L^{\hf_I}_n}$ holds by definition and $L^{\hf_I}_n$ is a
homogeneous ideal of $\kk[x_0, x_1, \dotsc, x_n]$; see
\cite[\S~II]{Macaulay--1927} or
\cite[Proposition~2.21]{Miller--Sturmfels--2005}.  More importantly,
the \emph{\bfseries (saturated) lexicographic ideal} for the Hilbert
polynomial $\hp_I$ is the monomial ideal
\[
L^{\hp_I}_n := \bigl( L^{\hf_I}_n : \langle x_0, x_1, \dotsc, x_{n}
\rangle^{\infty} \bigr) = \bigcup_{j \ge 1} \bigl\{ f \in \kk[x_0,
  x_1, \dotsc, x_n] \mid f \langle x_0, x_1, \dotsc, x_{n} \rangle^{j}
\subseteq L^{\hf_I}_n \bigr\}.
\]
Saturation with respect to the irrelevant ideal $\langle x_0, x_1,
\dotsc, x_{n} \rangle \subset \kk[x_0, x_1, \dotsc, x_n]$ does not
affect the Hilbert function in large degrees, so $L^{\hp_I}_n$ also
has Hilbert polynomial $\hp_I$.  From here on, we essentially always
work with saturated lexicographic ideals and point out when this is
not the case.

Given a finite sequence of nonnegative integers $a_0, a_1, \dotsc,
a_{n-1} \in \NN$, consider the monomial ideal $L(a_0, a_1, \dots,
a_{n-1}) \subset \kk[x_0, x_1, \dotsc, x_n]$ from
\cite[Notation~1.2]{Reeves--Stillman--1997} with generators
\[
\langle x_0^{a_{n-1}+1}, x_0^{a_{n-1}} x_1^{a_{n-2}+1}, \dotsc,
x_{0}^{a_{n-1}} x_{1}^{a_{n-2}} \dotsb x_{n-3}^{a_{2}} x_{n-2}^{a_{1}
  + 1}, x_{0}^{a_{n-1}} x_{1}^{a_{n-2}} \dotsb x_{n-2}^{a_{1}}
x_{n-1}^{a_{0}} \rangle .
\]
Lemma~\ref{lem:lexmingens}(i) appears in
\cite[Theorem~2.23]{Moore--2012}.

\begin{lemma}
  \label{lem:lexmingens}
  Let $\hp(t) = \sum_{i=0}^{d} \binom{t + i}{i + 1} - \binom{t + i -
    e_i}{i + 1}$, for integers $e_0 \ge e_1 \ge \dotsb \ge e_d > 0$,
  and let $n \in \NN$ satisfy $n > d = \deg \hp$.
  \begin{enumerate}
  \item Define $e_i = 0$, for $d+1 \le i \le n$, and $a_{j} = e_{j} -
    e_{j+1}$, for all $0 \le j \le n-1$.  We have
    \begin{align*}
      L^{\hp}_n &= L(a_0, a_1, \dotsc, a_{n-1}) \\
                &= \langle x_0, x_1, \dotsc , x_{n-d-2},
                  x_{n-d-1}^{a_{d} + 1}, \\
                &\relphantom{= \langle} x_{n-d-1}^{a_{d}}
                  x_{n-d}^{a_{d-1} + 1}, \dotsc ,
                  x_{n-d-1}^{a_{d}} x_{n-d}^{a_{d-1}} \dotsb
                  x_{n-3}^{a_{2}} x_{n-2}^{a_{1} + 1}, 
                  x_{n-d-1}^{a_{d}}
                  x_{n-d}^{a_{d-1}} \dotsb x_{n-2}^{a_{1}}
                  x_{n-1}^{a_{0}} \rangle .
    \end{align*}

  \item If there is an integer $0 \le \ell \le d-1$ such that $a_j =
    0$ for all $j \le \ell$, and $a_{\ell + 1} > 0$, then the minimal
    monomial generators of $L^{\hp}_n$ are given by $m_1, m_2, \dotsc,
    m_{n - \ell - 1}$, where
    \begin{align*}
      m_i &= x_{i-1}, \text{ for all } 1 \le i \le n-d-1, \\
      m_{n-d + k} &= \left( \prod_{j=0}^{k-1} x_{n-d-1 + j}^{a_{d -
          j}} \right) x_{n-d-1 + k}^{a_{d-k} + 1}\,, \text{ for all } 0 \le k \le d -
        \ell -2, \text{ and } \\
      m_{n-\ell-1} &= \prod_{j=0}^{d-\ell-1} x_{n-d-1 + j}^{a_{d - j}}.
    \end{align*}
    If $a_0 \ne 0$, then the minimal monomial generators are those
    listed in (i).
  \end{enumerate}
\end{lemma}

\begin{proof} $\;$
  \begin{enumerate} 
  \item It is straightforward to check that the ideal $L = L(a_0, a_1,
    \dotsc, a_{n-1})$ is saturated and lexicographic.  To see that $L$
    has the correct Hilbert polynomial, one can first prove that $L' =
    L(0, \dotsc, 0, a_d, 0, \dotsc, 0)$ has Hilbert polynomial
    $\sum_{i = 0}^d \binom{t+i}{i+1} - \binom{t+i-a_d}{i+1}$.  The
    general case then follows by induction on $d = \deg \hp$ and the
    short exact sequence
    \[ 
    0 \to \left( \kk[x_0, x_1, \dotsc, x_n] / L'' \right)(-a_d) \to
    \kk[x_0, x_1, \dotsc, x_n] / L \to \kk[x_0, x_1, \dotsc, x_n] / L'
    \to 0,
    \]
    where $L'' = L(a_0, a_1, \dotsc, a_{d-1}, 0, 0, \dotsc, 0)$ and
    the injection sends $1 \mapsto x_{n-d-1}^{a_{d}}$.

  \item We know that $x_{n-d-1}^{a_{d}} x_{n-d}^{a_{d-1}} \dotsb
    x_{n-2}^{a_{1}} x_{n-1}^{a_{0}} = x_{n-d-1}^{a_{d}}
    x_{n-d}^{a_{d-1}} \dotsb x_{n-\ell-3}^{a_{\ell+2}}
    x_{n-\ell-2}^{a_{\ell+1}}$, because either $a_0 = a_1 = \dotsb =
    a_{\ell} = 0$, or $a_0 \ne 0$ and $\ell = -1$.  If $\ell \ge 0$,
    then the monomial generators
    \begin{align*}
      x_{n-d-1}^{a_{d}} x_{n-d}^{a_{d-1}} \dotsb
      x_{n-\ell-3}^{a_{\ell+2}} x_{n-\ell-2}^{a_{\ell+1}+1}, \quad
      x_{n-d-1}^{a_{d}} x_{n-d}^{a_{d-1}} \dotsb
      x_{n-\ell-2}^{a_{\ell+1}} x_{n-\ell-1}^{a_{\ell}+1},\quad
      \dotsc,\qquad \\
      \hfill x_{n-d-1}^{a_{d}} x_{n-d}^{a_{d-1}} \dotsb
      x_{n-3}^{a_{2}} x_{n-2}^{a_{1}+1}
    \end{align*}
    from (i) are redundant, as they are multiples of the last monomial
    generator.  Removing these gives the monomial generators $m_1,
    m_2, \dotsc, m_{n-\ell-1}$.  For all $1< i < n-\ell$ and all $0 <
    j < i$ there exists $x_k$ dividing $m_j$ to higher order than the
    order to which it divides $m_i$ and minimality follows. \qedhere
  \end{enumerate}
\end{proof}

The nonminimal list of generators in Lemma~\ref{lem:lexmingens}(i) is
useful for describing operations on lexicographic ideals in terms of
Macaulay--Hartshorne and Gotzmann partitions.  Importantly,
Lemma~\ref{lem:lexmingens} shows that all sequences $(a_0, a_1,
\dotsc, a_{n-1})$ of nonnegative integers determine a lexicographic
ideal.

The next example uses Lemma~\ref{lem:lexmingens} to identify minimal
monomial generators.

\begin{example}
  \label{eg:twistedlexmingens} 
  The twisted cubic $X \subset \PP^3$ has $\hp_X(t) = \left[ \binom{t
      + 0}{0 + 1} - \binom{t + 0 - 4}{0 + 1} \right] + \left[ \binom{t
      + 1}{1 + 1} - \binom{t + 1 - 3}{1 + 1} \right]$, with ideal
  $L^{3t + 1}_3 \subset \kk[x_0, x_1, x_2, x_3]$.  We have $d = 1$ and
  $(e_0, e_1, e_2, e_3) = (4, 3, 0, 0)$, so that $(a_0, a_1, a_2) =
  (1, 3, 0)$.  Applying Lemma~\ref{lem:lexmingens} yields $L^{3t +
    1}_3 = L(1, 3, 0) = \langle x_0, x_1^4, x_1^3 x_2 \rangle$.
\end{example}

By analogy with $\plus$, we define the \emph{\bfseries lex-expansion}
of $L^{\hp}_n = L(a_0, a_1, \dotsc, a_{n-1})$ to be the lexicographic
ideal $\plus \bigl( L^{\hp}_n \bigr) := L(a_0 + 1, a_1, a_2, \dotsc,
a_{n-1})$.

\begin{lemma} 
  \label{lem:lexexpansion}
  Let $\hp$ be an admissible Hilbert polynomial and $n > \deg \hp$ a
  positive integer.  We have $\plus \bigl( L^{\hp}_n \bigr) = L^{\plus
    (\hp)}_n$ and the mapping $\plus$ on lexicographic ideals
  preserves codimension.
\end{lemma}

\begin{proof}
  See Lemma~\ref{lem:lexmingens}(i).  Note that $n - \deg \plus (\hp)
  = n - \deg \hp$.
\end{proof}

In analogy with $\lift$, for any ideal $I \subseteq \kk[x_0, x_1,
  \dotsc, x_n]$, we denote the \emph{\bfseries extension} ideal by
$\lift ( I ) = I \cdot \kk[x_0, x_1, \dotsc, x_{n+1}]$.  The following
is similar to Lemma~\ref{lem:lexexpansion}.

\begin{proposition} 
  \label{prop:lexextension}
  Let $\hp$ be an admissible Hilbert polynomial and $n > \deg \hp$ an
  integer.  We have $\lift \bigl( L^{\hp}_n \bigr) = L^{\lift
    (\hp)}_{n+1}$.  Equivalently $\lift \bigl( L(a_0, a_1, \dotsc,
  a_{n-1}) \bigr) = L(0, a_0, a_1, \dotsc, a_{n-1})$ holds, for all
  $a_0, a_1, \dotsc, a_{n-1} \in \NN$, and extension preserves
  codimension.
\end{proposition}

\begin{proof}
  This follows by relabelling $a_{0}' = 0$ and $a_{i}' = a_{i-1}$, for
  $0 < i \le d+1$, on generators.
\end{proof}

The \emph{\bfseries lexicographic point} of $\hilb^{\hp}(\PP^n)$ is
the point $\big[X_{L^{\hp}_n}\big]$, where $X_{I} \subseteq \PP^n$ denotes the
closed subscheme with saturated ideal $I$.  The lexicographic point is
nonsingular and lies on a unique irreducible component
$\hilb^{\hp}(\PP^n)$ called the \emph{\bfseries lexicographic
  component} \cite{Reeves--Stillman--1997}.  We can now describe a
tree structure on the set of Hilbert schemes.  We regard this as a
rough chart of the geography of Hilbert schemes, developed further in
Section~\ref{ch:Hilbertirred}.

\begin{theorem}
  \label{thm:hilbtree}
  For each positive codimension $c \in \ZZ$, the graph $\hilbtree_c$
  whose vertex set consists of all nonempty Hilbert schemes
  $\hilb^{\hp}(\PP^n)$ parametrizing codimension $c$ subschemes and
  whose edges are all pairs $\left( \hilb^{\hp}(\PP^n),
  \hilb^{\plus(\hp)}(\PP^n) \right)$ and $\left( \hilb^{\hp}(\PP^n),
  \hilb^{\lift(\hp)}(\PP^{n+1}) \right)$, where $\hp$ is an admissible
  Hilbert polynomial and $n = c + \deg \hp$, is an infinite full
  binary tree.  The root of the tree $\hilbtree_c$ is the Hilbert
  scheme $\hilb^1(\PP^c) = \PP^c$.
\end{theorem}

\begin{proof}
  Each pair of an admissible Hilbert polynomial $\hp$ and positive $c
  \in \ZZ$ uniquely determines $L = L^{\hp}_{c + \deg \hp}$ and the
  Hilbert scheme $\hilb^{\hp}(\PP^{c + \deg \hp})$ containing $[X_L]$.
  Lemma~\ref{lem:lexexpansion} and Proposition~\ref{prop:lexextension}
  combined with Lemma~\ref{lem:lexmingens} show that the mapping $\hp
  \mapsto \hilb^{\hp}(\PP^{c + \deg(\hp)})$ is a graph isomorphism.
  The root is then $\hilb^1(\PP^c) = \PP^c$.
\end{proof}

For positive $c \in \ZZ$, we call the tree of
Theorem~\ref{thm:hilbtree} the \emph{\bfseries Hilbert tree of
  codimension $c$}, denoted $\hilbtree_c$.  We call the disjoint union
$\hilbtree = \bigsqcup_{c \in \NN, c > 0} \hilbtree_c$ the
\emph{\bfseries Hilbert forest}.

\section{Strongly Stable Ideals}
\label{ch:stronglystable}

This section reviews a well-known algorithm that generates saturated
strongly stable ideals and then examines analogues of $\lift$ and
$\plus$ for strongly stable ideals.

A monomial ideal $I \subseteq \kk[x_0, x_1, \dotsc , x_n]$ is
\emph{\bfseries strongly stable}, or \emph{\bfseries $0$-Borel}, if,
for all monomials $m \in I$, for all $x_j$ dividing $m$, and for all
$x_i \lexg x_j$, we have $x_j^{-1} m x_i \in I$.  In characteristic
$0$, this is equivalent to being \emph{\bfseries Borel-fixed},
i.e.\ fixed by the action $\gamma \cdot x_j = \sum_{i=0}^{n}
\gamma_{ij} x_i$ of upper triangular matrices $\gamma \in
\operatorname{GL}_{n+1}(\kk)$
\cite[Proposition~2.7]{Bayer--Stillman--1987}.  For any monomial $m$,
let $\max m$ be the maximum index $j$ such that $x_j$ divides $m$, and
$\min m$ be the minimum such index.

\begin{example}
  The monomial ideal $I = \langle x_0^2, x_0 x_1, x_1^2 \rangle
  \subset \kk[x_0, x_1, x_2]$ is strongly stable.  The monomial $m =
  x_1^5 x_2 x_7^2 \in \kk[x_0, x_1, \dotsc, x_{13}]$ satisfies $\max m
  = 7$ and $\min m = 1$.
\end{example}

For a monomial ideal $I$, let $G(I)$ denote its minimal set of
monomial generators.  We gather some useful properties of strongly
stable ideals in the following lemma.

\begin{lemma}
  \label{lem:SSproperties}
  Let $I \subseteq \kk[x_0, x_1, \dotsc, x_n]$ be a monomial ideal.
  \begin{enumerate}
  \item The ideal $I$ is strongly stable if and only if, for all $g
    \in G(I)$, for all $x_j$ dividing $g$, and for all $x_i \lexg
    x_j$, we have $g x_i x_j^{-1} \in I$.
    
  \item If, for all $g' \in G(I)$, for all $x_j$ dividing $g'$, and
    for all $x_i \lexg x_j$, we have $g' x_i x_j^{-1} \in I$, then,
    for all monomials $m \in I$, there exists a unique $g \in G(I)$
    and unique monomial $m' \in \kk[x_0, x_1, \dotsc, x_n]$ such that
    $m = g m'$ and $\max g \le \min m'$.
    
  \item If $I$ is a strongly stable ideal, then $I$ is saturated with
    respect to the irrelevant ideal $\langle x_0, x_1, \dotsc, x_n
    \rangle$ if and only if no minimal monomial generators of $I$ are
    divisible by $x_n$.
    
  \item If $I$ is strongly stable with constant Hilbert polynomial
    $\hp_I \in \NN$, then there exists an integer $k \in \NN$ such
    that $x_{n-1}^k \in I$.
  \end{enumerate}
\end{lemma}

\begin{proof} $\;$
  \begin{enumerate}
  \item If $I$ is strongly stable, then this holds for all $g \in I$.
    Conversely, let $m \in I$ be any monomial.  By (ii), there is a
    unique factorization $m = g m'$, where $g \in G(I)$ and $m' \in
    \kk[x_0, x_1, \dotsc, x_n]$ is a monomial such that $\max g \le
    \min m'$.  Let $x_j$ divide $m$ and let $x_i \lexg x_j$.  Either
    $x_j$ divides $g$, in which case $g x_i x_j^{-1} \in I$ and $m x_i
    x_j^{-1} \in I$, or $x_j$ divides $m'$, in which case $m x_i
    x_j^{-1}$ is a multiple of $g$.

  \item See the proof of \cite[Lemma~2.11]{Miller--Sturmfels--2005}.

  \item If $x_n$ divides a minimal generator $g \in G(I)$, then for
    all $x_j$ we have $(g x_n^{-1}) x_j \in I$, while $g x_n^{-1}
    \notin I$.  Conversely, any monomial $m \in (I : \langle x_0, x_1,
    \dotsc, x_{n} \rangle) \setminus I$ yields a minimal monomial
    generator $m x_n \in G(I)$, by (ii).
  
  \item See the proof of \cite[Lemma~3.17]{Moore--2012}.  \qedhere
  \end{enumerate}
\end{proof}

Strongly stable ideals are generated by expansions.  Let $I \subseteq
\kk[x_0, x_1, \dotsc, x_n]$ be a saturated strongly stable ideal.  A
generator $g \in G(I)$ is \emph{\bfseries expandable} if there are no
minimal monomial generators of $I$ in the set $\left\{ x_i^{-1} g
x_{i+1} \mid x_i \text{ divides } g \text{ and } 0 \le i < n-1
\right\}$.  The \emph{\bfseries expansion} of $I$ at an expandable
generator $g$ is the monomial ideal
\[ 
I' = \langle I \setminus \langle g \rangle \rangle + \langle g x_{j}
\mid \max g \le j \le n-1 \rangle \subset \kk[x_0, x_1, \dotsc, x_n] ;
\]
see \cite[Definition~3.4]{Moore--2012}.  The monomial $1 \in \langle 1
\rangle$ is vacuously expandable with expansion $\langle x_0, x_1,
\dotsc, x_{n-1} \rangle \subset \kk[x_0, x_1, \dotsc, x_{n}]$. Parts
(i) and (iii) of Lemma~\ref{lem:SSproperties} ensure that the
expansion of a saturated strongly stable ideal is again saturated and
strongly stable.

\begin{example}
  \label{eg:lexexpansion}
  Let $L^{\hp}_n = L(a_0, a_1, \dotsc, a_{n-1}) = \langle m_1, m_2,
  \dotsc, m_{n - \ell - 1} \rangle$ be lexicographic, as in
  Lemma~\ref{lem:lexmingens}.  The last generator is $m_{n-\ell-1} =
  x_{n-d-1}^{a_{d}} x_{n-d}^{a_{d-1}} \dotsb x_{n-\ell-3}^{a_{\ell+2}}
  x_{n-\ell-2}^{a_{\ell+1}}$.  If $a_i > 0$, then $x_{n-i-1}$ divides
  $m_{n-\ell-1} x_{n-i} x_{n-i-1}^{-1}$ to order $a_i - 1$, which is
  not the case for any $m_j$.  Therefore, the expansion at
  $m_{n-\ell-1}$ has generators
  \[
  \{ m_1, m_2, \dotsc, m_{n-\ell-2}, m_{n-\ell-1} x_{n-\ell-2},
  m_{n-\ell-1} x_{n-\ell-1}, \dotsc, m_{n-\ell-1} x_{n-1} \}.
  \]
  These are easily verified to be minimal generators of $L(a_0 + 1,
  a_1, a_2, \dotsc, a_{n-1}) = \plus \bigl( L^{\hp}_n \bigr)$.
\end{example}

For a saturated strongly stable ideal $I \subseteq \kk[x_0, x_1,
  \dotsc, x_n]$, let $\nabla \left( I \right) \subseteq \kk[x_0, x_1,
  \dotsc, x_{n-1}]$ be the image of $I$ under the mapping $\kk[x_0,
  x_1, \dotsc, x_{n}] \to \kk[x_0, x_1, \dotsc, x_{n-1}]$, defined by
$x_j \mapsto x_j$ for $0 \le j \le n-2$ and $x_j \mapsto 1$ for $n-1
\le j \le n$.  This is the saturation of $I \cap \kk[x_0, x_1, \dotsc,
  x_{n-1}]$.

The following lemma generalizes Lemma~\ref{lem:lexexpansion} and
Proposition~\ref{prop:lexextension} to strongly stable ideals.

\begin{lemma}
  \label{lem:HpolyBorel}
  Let $I$ be a saturated strongly stable ideal.
  \begin{enumerate}
  \item If $I'$ is any expansion of $I$, then we have $\hp_{I'} =
    \plus (\hp_I)$.

  \item We have $\hp_{\nabla \left( I \right)} = \nabla (\hp_I)$.

  \item There exists $j \in \NN$ such that $\hp_{\lift ( I )} =
    \plus^j \lift (\hp_I)$.
  \end{enumerate}
\end{lemma}

\begin{proof} $\;$
  \begin{enumerate}
  \item Let $I'$ be the expansion of $I$ at $g$.  For all $d \ge \deg
    g$, Lemma~\ref{lem:SSproperties}(ii) shows that the only monomial
    in $I_d \setminus I'_d$ is $g x_n^{d - \deg g}$, so that
    $\hf_{I'}(d) = 1 + \hf_I(d)$.  Hence, we have $\hp_{I'} = 1 +
    \hp_I$.

  \item Let $S = \kk[x_0, x_1, \dotsc, x_n]$ and $J = I \cap \kk[x_0,
    x_1, \dotsc, x_{n-1}]$.  Because $I$ is saturated, $x_n$ is not a
    zero-divisor and
  \[ 
  0 \longrightarrow \left( S/I \right)(-1) \longrightarrow S/I
  \longrightarrow \kk[x_0, x_1, \dotsc, x_{n-1}] / J \longrightarrow 0
  \]
  is a short exact sequence.  The Hilbert function of $J$ now
  satisfies $\hf_{J}(i) = \hf_I(i) - \hf_I(i-1)$, for all $i \in \ZZ$.
  Hence, by saturating $J$ with respect to $\langle x_0, x_1, \dotsc,
  x_{n-1} \rangle$, we find that $\hp_{\nabla \left( I \right)}(t) =
  \hp_I(t) - \hp_I(t-1) = \left[ \nabla (\hp_I) \right](t)$.

  \item The ideal $\lift ( I )$ is saturated strongly stable by
    Lemma~\ref{lem:SSproperties}(i),(iii) with no minimal monomial
    generators divisible by $x_n$.  Thus, $\nabla \bigl( \lift ( I )
    \bigr) = I$.  Part (ii) shows that $\nabla (\hp_{\lift ( I )}) =
    \hp_I$, so that $\lift \nabla (\hp_{\lift ( I )}) = \lift
    (\hp_I)$.  Lemma~\ref{lem:liftproperties}(iii) then shows
    $\hp_{\lift ( I )} - \lift \nabla (\hp_{\lift ( I )}) \in
    \NN$. \qedhere
  \end{enumerate}
\end{proof}

\begin{table}
  \caption{Summary of Basic Operations}
  \centering
  {\renewcommand{\arraystretch}{1.5}
  \begin{tabular}{|p{.42\textwidth}|p{.53\textwidth}|}
    \hline
    
    \multirow{2}{*}{\parbox{.4\textwidth}{$\relphantom{www}$
        \\ Hilbert polynomial $\hp$ with partitions \\ $(b_1, b_2,
        \ldots, b_r)$ and $(e_0, e_1, \ldots, e_d)$}} &

    $\plus(\hp) = 1+\hp$ with partitions $(b_1, b_2, \ldots, b_r, 0)$
    and $(e_0+1, e_1, e_2, \ldots, e_d)$ \\ \cline{2-2}

    & $\lift(\hp)$ with partitions $(b_1+1, b_2+1, \ldots, b_r+1)$ and
    $(e_0, e_0, e_1, \ldots, e_d)$ \\ \hline

    \multirow{2}{*}{lex ideal $L^{\hp}_n = L(a_0, a_1, \ldots,
      a_{n-1})$} &
    
    $\plus(L^{\hp}_n) = L^{\plus(\hp)}_n = L(a_0+1, a_1, a_2, \ldots,
    a_{n-1})$ \\ \cline{2-2}

    & $\lift (L^{\hp}_n) = L^{\lift(\hp)}_{n+1} = L(0, a_0, a_1,
    \ldots, a_{n-1})$ \\ \hline
    
    \multirow{2}{*}{saturated str.\ st.\ ideal $I$ with $\hp_I = \hp$} &

    expansion $I'$ of $I$ with $\hp_{I'} = \plus(\hp)$ \\ \cline{2-2}

    & extension $\lift(I) = I \cdot \kk[x_0, x_1, \ldots, x_{n+1}]$
    \\ \hline
  \end{tabular}}
\end{table}

The heart of Reeves' algorithm \cite{Reeves--1992} is the following.

\begin{theorem}
  \label{thm:algorithm}
  If $I$ is a saturated strongly stable ideal of codimension $c$, then
  there is a finite sequence $I_{(0)}, I_{(1)}, \ldots, I_{(i)}$ such
  that $I_{(0)} = \langle 1 \rangle = \kk[x_0, x_1, \dotsc, x_c]$,
  $I_{(i)} = I$, and $I_{(j)}$ is an expansion or extension of
  $I_{(j-1)}$, for all $1 \le j \le i$.
\end{theorem}

\begin{proof}
  See \cite[Theorem~3.20]{Moore--2012} or
  \cite[Theorem~4.4]{Moore--Nagel--2014}.
\end{proof}

When $I \neq \langle 1 \rangle$, the first step is always to expand
$\langle 1 \rangle$ to $\langle x_0 , x_1, \ldots, x_{c-1} \rangle$.
We will start after this step, as we assume $\hp \neq 0$.  The
sequences of expansions and extensions are not generally unique, but
Theorem~\ref{thm:hilbtree} shows that they are for lexicographic
ideals.  Theorem~\ref{thm:algorithm} leads to the following algorithm.
The original is in \cite[Appendix~A]{Reeves--1992}, but we follow
\cite{Moore--2012, Moore--Nagel--2014}; see also
\cite[Section~5]{Cioffi--Lella--Marinari--Roggero--2011}.

\begin{algorithm}
  \label{alg:SSS}
  \qquad
  \begin{tabbing}
    \qquad \= \qquad \= \qquad \= \kill
    \textrm{Input:} an admissible Hilbert polynomial $\hp \in \QQ[t]$ and
    $n \in \NN$ satisfying $n > \deg \hp$ \\

  \textrm{Output:} all saturated strongly stable ideals with Hilbert
  polynomial $\hp$ in $\kk[x_0, x_1, \dotsc, x_n]$ \\[0.4em]
    
    $j = 0$; $d = \deg \hp$; \\[0.1em] 

    $\hq_0 = \nabla^d(\hp)$; $\hq_{1} = \nabla^{d-1} (\hp)$; \ldots
    ; $\hq_{d-1} = \nabla^{1} (\hp)$; $\hq_d = \hp$; \\[0.1em]

    $\mathcal{S} = \{ \langle x_0, x_1, \ldots, x_{n-d-1} \rangle \}$,
  where $\langle x_0, x_1, \ldots, x_{n-d-1} \rangle \subset\kk[x_0,
    x_1, \dotsc, x_{n-d}]$; \\[0.2em]

    \textrm{WHILE} $j \le d$ \textrm{DO} \\

    \>  $\mathcal{T} = \varnothing$; \\

    \>  \textrm{FOR} $J \in \mathcal{S}$, considered as an ideal in
    $\kk[x_0, x_1, \dotsc, x_{n-d+j}]$ \textrm{DO} \\

    \> \>    \textrm{IF} $\hq_j - \hp_J \ge 0$ \textrm{THEN} \\

    \> \> \> compute all sequences of $\hq_j - \hp_J$ expansions that
    begin with $J$; \\
    \> \> \> $\mathcal{T} = \mathcal{T} \cup \text{ the resulting set of
      sat.\ str.\ st.\ ideals with Hilbert polynomial } \hq_j$; \\

    \>  $\mathcal{S} = \mathcal{T}$; \\
    \> $j = j+1$; \\

    \textrm{RETURN} $\mathcal{S}$
  \end{tabbing}
\end{algorithm}

\begin{proof}
  See \cite[Algorithm~3.22]{Moore--2012}.  Here, $\mathcal{S}$ is
  reset at the $j$th step to Moore's $\mathcal{S}^{(d-j)}$.
\end{proof}

\begin{example}
  \label{eg:algorithm}
  To compute all codimension $2$ saturated strongly stable ideals with
  Hilbert polynomial $\hp(t) = 3t + 1$, we first compute $\nabla(\hp)
  = 3$.  We produce all length $2$ sequences of expansions beginning
  at $\langle x_0, x_1 \rangle \subset \kk[x_0, x_1, x_2]$.  The only
  expandable generator of $\langle x_0, x_1 \rangle$ is $x_1$, with
  expansion $\langle x_0, x_1^2 \rangle$.  Both $x_0$ and $x_1^2$ are
  expandable in $\langle x_0, x_1^2 \rangle$, with expansions $\langle
  x_0^2, x_0 x_1, x_1^2 \rangle, \langle x_0, x_1^3 \rangle \subset
  \kk[x_0, x_1, x_2]$.  Extending each of these to $\kk[x_0, x_1, x_2,
    x_3]$, their Hilbert polynomials are $3t + 1$ and $3t$,
  respectively.  Thus, we make all possible expansions of $\langle
  x_0, x_1^3 \rangle$; expansion at $x_0$ gives $\langle x_0^2, x_0
  x_1, x_0 x_2, x_1^3 \rangle$, and expansion at $x_1^3$ gives
  $\langle x_0, x_1^4, x_1^3 x_2 \rangle$.  Hence, there are three
  codimension $2$ saturated strongly stable ideals with Hilbert
  polynomial $3t + 1$ namely, $\langle x_0^2, x_0 x_1, x_1^2 \rangle$,
  $\langle x_0^2, x_0 x_1, x_0 x_2, x_1^3 \rangle$, and $\langle x_0,
  x_1^4, x_1^3 x_2 \rangle$ in $\kk[x_0, x_1, x_2, x_3]$.
\end{example}

\section{K-Polynomials and Climbing Trees}
\label{ch:Kpolynomials}

The goal of this section is to understand where Hilbert functions and
Hilbert polynomials of saturated strongly stable ideals coincide.
Theorem~\ref{thm:Kpolydegree} states that among the saturated strongly
stable ideals with a fixed codimension and Hilbert polynomial, the
degree of the $K$-polynomial of the lexicographic ideal is
\emph{strictly} the largest.  The proof tracks the genesis of minimal
monomial generators as Algorithm~\ref{alg:SSS} traces the path from
$1$ to $\hp$ in $\Macaulaytree$.  Proposition~\ref{prop:Kpolycommence}
identifies where the inequality first occurs and
Proposition~\ref{prop:Kpolypersist} shows that it persists.

The \emph{\bfseries Hilbert series} of a finitely generated graded
$\kk[x_0, x_1, \dotsc, x_n]$-module $M$ is the formal power series
$\hs_M( T ) = \sum_{i \in \ZZ} \hf_{M}(i) \, T^i \in \ZZ[T^{-1}][\![ T
  ]\!]$.  The Hilbert series of $M$ is a rational function $\hs_M( T )
= (1-T)^{-n-1} \hk_M( T )$ and the \emph{\bfseries $K$-polynomial} of
$M$ is the numerator $\hk_M$, possibly divisible by $1-T$, of $\hs_M$;
see \cite[Theorem~8.20]{Miller--Sturmfels--2005}.  For a quotient by a
homogeneous ideal $I$, we use the notation $\hs_I = \hs_{\kk[x_0, x_1,
    \dotsc, x_n]/I}$ and $\hk_I = \hk_{\kk[x_0, x_1, \dotsc, x_n]/I}$.

We consider a fundamental example.

\begin{example}
  \label{eg:Kpoly}
  If $S = \kk[x_0, x_1, \dotsc, x_n]$, then we have $\hk_{S}(T) = 1$,
  as $\hs_{S}(T) = (1 - T)^{-n-1}$.  If $d \in \NN$, then we have
  $\hs_{S(-d)}(T) = (1 - T)^{-n-1} T^d$ and $\hk_{S(-d)}(T) = T^d$.
\end{example}

The following well-known lemma is useful.  As before, $G(I)$ denotes
the minimal set of monomial generators of a monomial ideal $I$.

\begin{lemma}
  \label{lem:Kpolydefbnd}
  Let $I \subseteq \kk[x_0, x_1, \dotsc, x_n]$ be a strongly stable
  ideal.
  \begin{enumerate}
  \item We have $\hk_{I}(T) = 1 - \sum_{g \in G(I)} T^{\, \deg g} (1 -
    T)^{\max g}$.
    
  \item We have $\deg \hk_{I} \le \max_{g \in G(I)} \bigl\{ \deg g +
    \max g \bigr\}$.
  \end{enumerate}
\end{lemma}

\begin{proof} $\;$
  \begin{enumerate}
  \item This follows by Lemma~\ref{lem:SSproperties}(ii) and
    Example~\ref{eg:Kpoly}; also see
    \cite[Proposition~2.12]{Miller--Sturmfels--2005}.
  
  \item This follows immediately from (i).  \qedhere
  \end{enumerate}
\end{proof}

Let $\deg \hs_{I} := \deg \hk_{I} - n-1$.  The next lemma establishes
that $\deg \hs_{I}$ is the maximal value where $\hf_I$ and $\hp_I$
differ.

\begin{lemma} 
  \label{lem:Kpolycoincide}
  Let $I \subseteq \kk[x_0, x_1, \dotsc, x_n]$ be a homogeneous ideal
  with rational Hilbert series
  $\hs_I( T ) = \sum_{i \in \ZZ} \hf_I(i) \, T^i = \hk_I( T
  )(1-T)^{-n-1}$
  and Hilbert polynomial $\hp_I$.  We have $\hf_I(i) = \hp_I(i)$ for
  all $i > \deg \hs_I$, while $\hf_I(i) \ne \hp_I(i)$ for
  $i = \deg \hs_I$.
\end{lemma}

\begin{proof}
  Let $\hk_I(T) = \sum_{k=0}^d c_k T^k \in \ZZ[T]$.  Expanding
  $(1-T)^{-n-1}$ and gathering terms yields
  \[
  \hs_I(T) = \sum_{i \in \NN} \left( c_0 \binom{n+i}{n} + c_1
  \binom{n+i-1}{n} + \dotsb + c_d \binom{n+i-d}{n} \right) T^i.
  \]
  Thus, the Hilbert function equals $\hf_I(i) = c_0 \binom{n+i}{n} +
  c_1 \binom{n+i-1}{n} + \dotsb + c_d \binom{n+i-d}{n}$, \emph{for
    all} $i \in \ZZ$, and the Hilbert polynomial is $\hp_I(t) = c_0
  \binom{t + n}{n} + c_1 \binom{t + n-1}{n} + \dotsb + c_d \binom{t +
    n-d}{n}$.  Following our convention in Remark~\ref{rmk:binconv},
  the equality $\binom{n+i-j}{n} = \left. \binom{t + n-j}{n}
  \right\rvert_{t=i}$ holds if and only if $i \ge -n+j$.  This implies
  that $\hf_I(i) = \hp_I(i)$ whenever $i \ge -n+d = 1 + \deg \hs_I$,
  proving the first statement.  To finish, set $i = d - n-1 = \deg
  \hs_I$ and compare the value
  \begin{align*}
    \hp_I(i) &= \sum_{j=0}^d c_j \left. \binom{t + n-j}{n}
               \right\rvert_{t=i} = \sum_{j=0}^{d-1} c_j
               \left. \binom{t + n-j}{n} \right\rvert_{t=i} + c_d
               \left. \binom{t + n-d}{n} \right\rvert_{t=i} \\
             &= \sum_{j=0}^{d-1} c_j \left. \binom{t + n-j}{n}
               \right\rvert_{t=i} + c_d (-1)^n 
  \end{align*}
  with the value
  $\hf_I(i) = \sum_{j=0}^{d-1} c_j \binom{i+n-j}{n} + c_d \cdot 0$.
  As $c_d \ne 0$, we are finished.
\end{proof}

In the next two propositions, let $L^{\hp}_n = \langle m_1, m_2,
\dotsc, m_{n-\ell-1} \rangle$, as in Lemma~\ref{lem:lexmingens}.
These propositions examine the behaviour of $\deg \hk_I$ for saturated
strongly stable ideals $I$.

\begin{proposition} 
  \label{prop:Kpolycommence}
  Let $L^{\hp}_n \subset \kk[x_0, x_1, \dotsc , x_n]$ be any
  lexicographic ideal.
  \begin{enumerate}
  \item If $m \in G(L^{\hp}_n)$ is a minimal monomial generator, then
    $m$ is expandable if and only if $m$ is the smallest minimal
    monomial generator of its degree. 
    
  \item Let $m_{n - \ell - 1}$ denote the last minimal monomial
    generator of $L^{\hp}_n$.  If $m \neq m_{n - \ell - 1}$ is any
    other expandable generator of $L^{\hp}_n$ and $(L^{\hp}_n)'$ is
    the expansion of $L^{\hp}_n$ at $m$, then every minimal monomial
    generator $g \in G((L^{\hp}_n)')$ satisfies $\deg g < 1 + \deg
    m_{n - \ell - 1}$.
    
  \item Moreover, in (ii), we have $\deg \hk_{L^{\plus (\hp)}_n} >
    \deg \hk_{(L^{\hp}_n)'}$.
\end{enumerate}
\end{proposition}

\begin{proof} $\;$
  \begin{enumerate}
  \item By inspection, if $\deg m_j = \deg m_{j+1}$ holds, then $m_j$
    is not expandable.

  \item By (i), we have $\deg m < \deg m_{n-\ell-1}$.  But the minimal
    generators of $(L^{\hp}_n)'$ are
    \[ 
    G((L^{\hp}_n)') = \bigl( \{ m_1, m_2, \dotsc, m_{n-\ell-1} \}
    \setminus \{ m \} \bigr) \cup \bigl\{ m x_{\max m}, m x_{\max m
      +1}, \dotsc, m x_{n-1} \bigr\},
    \]
    and $\deg m_j$ is maximized at $j = n-\ell-1$, which gives the
    inequality.

  \item We know $L^{\plus (\hp)}_n = \langle m_1, m_2, \dotsc,
    m_{n-\ell-2}, m_{n-\ell-1} x_{n-\ell-2}, m_{n-\ell-1}
    x_{n-\ell-1}, \dotsc, m_{n-\ell-1} x_{n-1} \rangle$, by
    Example~\ref{eg:lexexpansion}.  As $\deg m_j$ is maximized at
    $m_{n-\ell-1}$, we get $\deg \hk_{L^{\plus (\hp)}_n} = \deg
    m_{n-\ell-1} + n$.  Then (ii) and Lemma~\ref{lem:Kpolydefbnd}(ii)
    yield the desired inequality. \qedhere
  \end{enumerate}
\end{proof}

Proposition \ref{prop:Kpolypersist} explains how the $K$-polynomial
inequality in Proposition~\ref{prop:Kpolycommence} persists.

\begin{proposition} 
  \label{prop:Kpolypersist}
  Let $I \subset \kk[x_0, x_1, \dotsc , x_n]$ be a saturated strongly
  stable ideal, $\hp = \hp_I$, and $m_{n-\ell-1}$ be the last minimal
  generator of $L^{\hp}_n$.  Consider the following condition on $I$:
  \[
  \text{all generators $g \in G(I)$ satisfy $\deg g < \deg
    m_{n-\ell-1}$ and $\max g \le \max m_{n-\ell-1}$.}
  \tag{$\star$} \label{star}
  \]
  If $I$ satisfies (\ref{star}), then the following are true:
  \begin{enumerate}
  \item $\deg \hk_{L^{\hp}_n} > \deg \hk_I$, or equivalently, $\deg
    \hs_{L^{\hp}_n} > \deg \hs_I$;

  \item if $I'$ denotes any expansion of $I$, then $I'$ satisfies
    (\ref{star}) with respect to $L^{\plus (\hp)}_n$;

  \item the extension $\lift ( I )$ satisfies (\ref{star}) with
    respect to $L^{\hp_{\lift (I)}}_{n+1}$; and

  \item if $I_{(0)}, I_{(1)}, \dotsc, I_{(i)}$ is any finite sequence
    such that $I_{(0)} = I$ and $I_{(j)}$ is an expansion or extension
    of $I_{(j-1)}$, for all $0 < j \le i$, then $I_{(i)}$ satisfies
    (\ref{star}).
  \end{enumerate}
\end{proposition}

\begin{proof} $\;$
  \begin{enumerate}
  \item This follows immediately from Lemma~\ref{lem:Kpolydefbnd}(ii)
    and (\ref{star}).
 
  \item The condition (\ref{star}) for the expansion $I'$ becomes:
    every generator $g' \in G(I')$ satisfies $\deg g' < 1 + \deg m_{n
      - \ell - 1}$ and $\max g' \le n-1$.  Both inequalities hold, by
    definition of the minimal monomial generators of $I'$ and because
    $I$ satisfies (\ref{star}).
  
  \item An analogous condition to (\ref{star}) holds between $\lift (
    I )$ and $L^{\lift (\hp)}_n$.  Replacing $L^{\lift (\hp)}_n$ by
    $L^{\plus^j \lift (\hp)}_n$, with $j$ defined by
    Lemma~\ref{lem:HpolyBorel}(iii), results in higher degree and
    maximum index of the last minimal generator of $L^{\plus^j \lift
      (\hp)}_n$; cf.\ Example~\ref{eg:lexexpansion}.  Hence, $\lift (
    I )$ satisfies (\ref{star}).

  \item We apply induction to $i$.  The case $i = 1$ is resolved by
    (ii) and (iii).  If $i > 1$, then (ii) and (iii) ensure that
    $I_{(1)}$ satisfies (\ref{star}), and we apply the induction
    hypothesis. \qedhere
  \end{enumerate}
\end{proof}

Combining Propositions~\ref{prop:Kpolycommence} and
\ref{prop:Kpolypersist} leads to the main result of this section.

\begin{theorem} 
  \label{thm:Kpolydegree}
  Let $I \subset \kk[x_0, x_1, \dotsc , x_n]$ be a saturated strongly
  stable ideal and denote $\hp = \hp_I$.  If $I \ne L^{\hp}_n$, then
  we have $\deg \hk_{L^{\hp}_n} > \deg \hk_{I}$, or equivalently,
  $\deg \hs_{L^{\hp}_n} > \deg \hs_{I}$.
\end{theorem}

\begin{proof}  
  Both $I$ and $L^{\hp}_n$ are generated by Algorithm~\ref{alg:SSS}.
  Let $c$ be their codimension and $I_{(1)}, I_{(2)}, \dotsc, I_{(i)}$
  be a finite sequence such that $I_{(1)} = \langle x_0, x_1, \dotsc,
  x_{c-1} \rangle \subset \kk[x_0, x_1, \dotsc, x_c]$, $I_{(i)} = I$,
  and $I_{(j)}$ is an expansion or extension of $I_{(j-1)}$, for all
  $1 < j \le i$.  Theorem~\ref{thm:hilbtree} implies that if $I \ne
  L^{\hp}_n$, then there is some $2 \le k \le i$ such that $I_{(j)}$
  is lexicographic, for all $1 \le j \le k-1$, but $I_{(k)}$ is not.
  By Proposition~\ref{prop:Kpolycommence}(i), $I_{(k)}$ is the
  expansion of $I_{(k-1)}$ at a minimal generator of nonmaximal
  degree.  Proposition~\ref{prop:Kpolycommence}(ii) then shows that
  $I_{(k)}$ satisfies (\ref{star}).  Applying
  Proposition~\ref{prop:Kpolypersist}(iv) to the subsequence $I_{(k)},
  I_{(k+1)}, \dotsc, I_{(i)}$ shows that $I_{(i)} = I$ satisfies
  (\ref{star}), hence, applying Proposition~\ref{prop:Kpolypersist}(i)
  finishes the proof.
\end{proof}

\begin{proof}[Proof of Theorem~\ref{thm:Kpolydegreeintro}]
  Theorem~\ref{thm:Kpolydegree} proves the claim.
\end{proof}

\begin{example}
  \label{eg:Kpolydegree}
  Example~\ref{eg:algorithm} shows the saturated strongly stable
  ideals in $\kk[x_0, x_1, x_2, x_3]$ with Hilbert polynomial $3t+1$
  are $\langle x_0^2, x_0 x_1, x_1^2 \rangle$, $\langle x_0^2, x_0
  x_1, x_0 x_2, x_1^3 \rangle$, and $L^{3t+1}_3 = \langle x_0, x_1^4,
  x_1^3 x_2 \rangle$.  Lemma~\ref{lem:Kpolydefbnd}(i) yields $\deg
  \hk_{L^{3t+1}_3} = 6$, $\deg \hk_{\langle x_0^2, x_0 x_1, x_0 x_2,
    x_1^3 \rangle} = 3$, and $\deg \hk_{\langle x_0^2, x_0 x_1, x_1^2
    \rangle} = 3$.
\end{example}

\begin{corollary}
  \label{cor:Kpolydegree}
  Let $I \subset \kk[x_0, x_1, \dotsc , x_n]$ be a saturated strongly
  stable ideal and let $\hp = \hp_I$.  If $I \ne L^{\hp}_n$, then
  there exists $k \in \ZZ$ such that $\hf_I(j) = \hp(j)$, for all $j
  \ge k$, but $\hf_{L^{\hp}_n}(k) \ne \hp(k)$.
\end{corollary}

\begin{proof}
  Lemma~\ref{lem:Kpolycoincide} and Theorem~\ref{thm:Kpolydegree} show
  that this is the case for $k = \deg \hs_{L^{\hp}_n}$.
\end{proof}

\section{The Ubiquity of Smooth Hilbert Schemes}
\label{ch:Hilbertirred}

The goal now is to investigate our proposed geography of Hilbert
schemes, formally described by the collection of trees $\hilbtree$, by
applying the Hilbert series inequalities of
Theorem~\ref{thm:Kpolydegree}.  Surprisingly, we recover a
not-so-well-known family of irreducible Hilbert schemes, first
discovered by Gotzmann \cite[Proposition~1]{Gotzmann--1989}.
Moreover, we observe that a complete classification of Hilbert schemes
with unique strongly stable ideals can be given by examining how
Reeves' algorithm interacts with $\hilbtree_c$.  These Hilbert schemes
are nonsingular and irreducible over algebraically closed or
characteristic $0$ fields, and under natural probability distributions
on the trees $\hilbtree_c$ occur with probability at least $0.5$.

The next two lemmas are used to prove the classification
Theorem~\ref{thm:SSSunique2}.

\begin{lemma}
  \label{lem:Hfunctionineq}
  Let $I \subset \kk[x_0, x_1, \dotsc, x_n]$ be a homogeneous ideal,
  and let $\hp = \hp_I$.
  \begin{enumerate}
  \item We have $\hf_I(i) \ge \hf_{L^{\hp}_n}(i)$, for all $i \in
    \ZZ$, where $L^{\hp}_n$ is the corresponding lexicographic ideal.

  \item The Hilbert function of $\lift ( I ) = I \cdot \kk[x_0, x_1,
    \dotsc, x_{n+1}]$ is given by $\hf_{\lift ( I )}(i) = \sum_{0 \le
    j \le i} \hf_I(j)$.
  \end{enumerate}
\end{lemma}

\begin{proof} $\;$
  \begin{enumerate}
  \item Section~\ref{ch:lexforest} defines the (possibly unsaturated)
    lexicographic ideal $L^{\hf}_n \subset \kk[x_0, x_1, \dotsc, x_n]$
    for the Hilbert function $\hf = \hf_I$.  We have
    \[
    \hf(i) = \dim_{\kk} \kk[x_0, x_1, \dotsc, x_n]_i / (L^{\hf}_n)_i
    \ge \dim_{\kk} \kk[x_0, x_1, \dotsc, x_n]_i / (L^{\hp}_n)_i,
    \]
    for all $i \in \ZZ$, because $L^{\hp}_n$ contains $L^{\hf}_n$.  It
    follows that $\hf(i) \ge \hf_{L^{\hp}_n}(i)$.

  \item The homogeneous piece $\left( \lift ( I ) \right)_i$ has
    decomposition
    \[
    \bigl( \lift ( I ) \bigr)_i = \bigoplus_{j \in \NN, j \le i} I_j
    \cdot x_{n+1}^{i - j} \subset \bigoplus_{j \in \NN, j \le i}
    \kk[x_0, x_1, \dotsc, x_n]_j \cdot x_{n+1}^{i - j} = \kk[x_0, x_1,
      \dotsc, x_{n+1}]_i
    \]
    and the desired equality follows directly.  \qedhere
  \end{enumerate}
\end{proof}

\begin{lemma}
  \label{lem:expandablegens}
  Let $c > 0$, $\hp$ be an admissible Hilbert polynomial, and
  $\Lambda$ be a finite sequence of $\lift$'s and $\plus$'s.  The
  number of expandable minimal monomial generators of $L^{\Lambda
    (\hp)}_{c + \deg \Lambda (\hp)}$ is greater than or equal to the
  corresponding number for $L^{\hp}_{c + \deg \hp}$.
\end{lemma}

\begin{proof}
  This follows from Proposition~\ref{prop:Kpolycommence}(i) and the
  definition of expandable.
\end{proof}

We can now prove our classification result.

\begin{theorem}
  \label{thm:SSSunique2}
  Let $\hp(t) = \sum_{j=1}^r \binom{t + b_j - j+1}{b_j}$, for $b_1 \ge
  b_2 \ge \dotsb \ge b_r \ge 0$.  The lexicographic ideal is the
  unique saturated strongly stable ideal of codimension $c$ with
  Hilbert polynomial $\hp$ if and only if at least one of the
  following holds: (i) $b_r > 0$, (ii) $c \ge 2$ and $r \le 2$, (iii)
  $c = 1$ and $b_1 = b_r$, or (iv) $c = 1$ and $r - s \le 2$, where
  $b_1 = b_2 = \dotsb = b_{s} > b_{s+1} \ge \dotsb \ge b_r$.
\end{theorem}

\begin{proof}
  Let $b_r > 0$.  Proposition~\ref{prop:Macaulaytree} gives $\hp =
  \lift^{b_r} \plus \lift^{b_{r-1} - b_r} \dotsb \plus \lift^{b_2 -
    b_3} \plus \lift^{b_1 - b_2} (1)$, so there exists $\hq$ such that
  $\hp = \lift (\hq)$.  Saturated strongly stable ideals are generated
  by Algorithm~\ref{alg:SSS}.  The procedure is recursive and
  generates the codimension $c$ saturated strongly stable ideals with
  Hilbert polynomial $\hp$ by extending all codimension $c$ saturated
  strongly stable ideals with Hilbert polynomial $\hq = \nabla \lift
  (\hq)$ and keeping those with Hilbert polynomial $\hp$.

  By Proposition~\ref{prop:lexextension}, we have $\lift (L^{\hq}_n) =
  L^{\hp}_{n+1} \subset \kk[x_0, x_1, \dotsc, x_{n+1}]$, where $n = c
  + \deg \hq$.  It suffices to prove the following statement:

  \bigskip 
  \noindent
  \emph{If $J \subset \kk[x_0, \dotsc, x_{n}]$ is a saturated,
    strongly stable, nonlexicographic ideal, then $\hp_{\lift ( J )}
    \ne \lift (\hp_J)$.}

  \bigskip 
  \noindent 
  Let $I = \lift ( J )$ be the extension of such an ideal, $\hq =
  \hp_J$, and $\hp = \lift (\hq)$.  By
  Lemma~\ref{lem:HpolyBorel}(iii), we must show that $\hp_{I} - \hp >
  0$.  Setting $d_{\hq} = \deg \hs_{L^{\hq}_n}$, we show that
  $\hf_{I}(i) > \hf_{L^{\hp}_{n+1}}(i)$, for all integers $i \ge
  d_{\hq}$.  Lemma~\ref{lem:Hfunctionineq}(ii) implies that
  \begin{align*}
    \hf_{I}(i) &= \sum_{0 \le j \le i} \hf_J(j) = \sum_{0 \le j \le
                 d_{\hq}} \hf_J(j) + \sum_{d_{\hq} < j \le i}
                 \hf_J(j) \text{ and
                 } \\
    \hf_{L^{\hp}_{n+1}}(i) &= \sum_{0 \le j \le i}
                             \hf_{L^{\hq}_n}(j) = \sum_{0
                             \le j \le d_{\hq}} \hf_{L^{\hq}_n}(j) +
                             \sum_{d_{\hq} < j \le i} \hf_{L^{\hq}_n}(j).
  \end{align*}
  Theorem~\ref{thm:Kpolydegree} implies $\deg \hs_J < d_{\hq}$, so
  that $\sum_{d_{\hq} < j \le i} \hf_J(j) = \sum_{d_{\hq} < j \le i}
  \hq(j) = \sum_{d_{\hq} < j \le i} \hf_{L^{\hq}_n}(j)$, by
  Lemma~\ref{lem:Kpolycoincide}.  We must prove that $\sum_{0 \le j
    \le d_{\hq}} \hf_J(j) > \sum_{0 \le j \le d_{\hq}}
  \hf_{L^{\hq}_n}(j)$.  Lemma~\ref{lem:Hfunctionineq}(i) gives
  $\sum_{0 \le j \le d_{\hq}} \hf_J(j) \ge \sum_{0 \le j \le d_{\hq}}
  \hf_{L^{\hq}_n}(j)$ and strict inequality fails if and only if
  $\hf_J(j) = \hf_{L^{\hq}_n}(j)$, for all $0 \le j \le d_{\hq}$.  But
  this contradicts Corollary~\ref{cor:Kpolydegree}, so strict
  inequality holds, and $\hp_{I} - \hp > 0$.  This settles the case
  $b_r > 0$.  To prove the remaining cases, we examine
  Algorithm~\ref{alg:SSS}.
    
  Let $c \ge 2$ and $b_r = 0$.  If $r = 1$, then $\hp = 1$ and
  uniqueness holds.  If $r = 2$, then $\hp = \plus \lift^{b_1} (1)$
  and to generate saturated strongly stable ideals with codimension
  $c$ and Hilbert polynomial $\hp$, we take $b_1$ extensions from
  $L^{1}_c = \langle x_0, x_1, \dotsc, x_{c-1} \rangle$ followed by
  one expansion.  The only expandable generator is $x_{c-1}$, hence
  uniqueness holds.  If $r \ge 3$, then consider $\plus \lift^{b_{1} -
    b_{2}} (1)$ and its lexicographic ideal $ \langle x_0, x_1,
  \dotsc, x_{c-2}, x_{c-1}^2, x_{c-1}x_{c}, \dotsc, x_{c-1} x_{c + b_1
    - b_2 - 1} \rangle$ in $\kk[x_0, x_1, \dotsc, x_{c + b_1 - b_2}]$.
  As $c \ge 2$, this ideal has two expandable generators, $x_{c-2}$
  and $x_{c-1} x_{c + b_1 - b_2 - 1}$.  Lemma~\ref{lem:expandablegens}
  then implies that $ L^{\lift^{b_{r-1}} \plus \lift^{b_{r-2} -
      b_{r-1}} \dotsb \plus \lift^{b_2 - b_3} \plus \lift^{b_1 - b_2}
    (1)}_{c + b_1}$ has at least two expandable generators, which give
  distinct saturated strongly stable ideals with Hilbert polynomial
  $\hp$ and codimension $c$.

  Let $c = 1$ and $b_r = 0$.  If $b_1 = b_r$, then $\hp = r$ and to
  get codimension $1$ saturated strongly stable ideals with Hilbert
  polynomial $\hp$, we take $r - 1$ expansions from $L^1_1 = \langle
  x_0 \rangle \subset \kk[x_0, x_1]$; the possibilities are $\langle
  x_0^2 \rangle, \langle x_0^3 \rangle, \dotsc, \langle x_0^{r}
  \rangle \subset \kk[x_0, x_1]$.  Let $b_1 = b_2 = \dotsb = b_{s} >
  b_{s+1} \ge \dotsb \ge b_r$.  If $r - s = 1$, then we have $\hp =
  \plus \lift^{b_{r-1}} \plus^{r-2} (1)$ and we take $b_{r-1}$
  extensions of $\langle x_0^{r-1} \rangle \subset \kk[x_0, x_1]$,
  followed by the unique expansion of $\langle x_0^{r-1} \rangle
  \subset \kk[x_0, x_1, \dotsc, x_{1+b_{r-1}}]$.  If $r - s = 2$, then
  we have $\hp = \plus \lift^{b_{r-1}} \plus \lift^{b_{r-2} - b_{r-1}}
  \plus^{r-3} (1)$.  We extend $b_{r-2} - b_{r-1}$ times from $\langle
  x_0^{r-2} \rangle \subset \kk[x_0, x_1]$, we expand to obtain
  $\langle x_0^{r-1}, x_0^{r-2} x_1, \dotsc, x_0^{r-2} x_{b_{r-2} -
    b_{r-1}} \rangle \subset \kk[x_0, x_1, \dotsc, x_{1 + b_{r-2} -
      b_{r-1}}]$, we take $b_{r-1}$ further extensions, and we expand
  at $x_0^{r-2} x_{b_{r-2} - b_{r-1}}$.  Now suppose $r - s \ge 3$.
  Consider the polynomial $ \plus \lift^{b_{s+1} - b_{s+2}} \plus
  \lift^{b_{s} - b_{s+1}} \plus^{s-1} (1) $ obtained from $\hp$ by
  truncation, with lexicographic ideal
  \[
  \langle x_0^{s+1}, x_0^{s} x_1, \dotsc, x_0^{s} x_{b_{s} - b_{s+1} -
    1}, x_0^{s} x_{b_{s} - b_{s+1}}^2, x_0^{s} x_{b_{s} - b_{s+1}}
  x_{b_{s} - b_{s+1} + 1}, \dotsc, x_0^{s} x_{b_{s} - b_{s+1}}
  x_{b_{s} - b_{s+2}} \rangle .
  \]
  As $b_s > b_{s+1}$, both $x_0^{s} x_{b_{s} - b_{s+1} - 1}$ and
  $x_0^{s} x_{b_{s} - b_{s+1}} x_{b_{s} - b_{s+2}}$ are expandable and
  Lemma~\ref{lem:expandablegens} shows that $ L^{\lift^{b_{r-1}} \plus
    \lift^{b_{r-2} - b_{r-1}} \plus \dotsb \plus \lift^{b_2 - b_3}
    \plus \lift^{b_1 - b_2} (1)}_{1 + b_1}$ has at least two distinct
  expansions.  \qedhere
\end{proof}

\begin{remark}
  The case $b_r > 0$ is a consequence of
  Theorem~\ref{thm:Kpolydegree}.  Another approach might exist using
  Stanley decompositions; see \cite{Maclagan--Smith--2005,
    Sturmfels--White--1991, Stanley--1982}.  Indeed,
  Proposition~\ref{prop:lexextension} follows by considering a Stanley
  decomposition of the lexicographic ideal, while
  Lemma~\ref{lem:SSproperties}(ii) gives a Stanley decomposition of
  $I$.  We thank D.~Maclagan for pointing this out.
\end{remark}

A \emph{\bfseries non-standard} Borel-fixed ideal is a Borel-fixed
ideal that is not strongly stable.  Such ideals exist only in positive
characteristic.  Pardue \cite[Chapter~2]{Pardue--1994} gives the
following combinatorial criterion for $I \subset \kk[x_0, x_1, \ldots,
  x_n]$ to be Borel-fixed when $\kk$ is infinite of characteristic $p
> 0$: $I$ is monomial and for all monomials $m \in I$, if $x_j^{\ell}
\Vert m$ and $x_i \lexg x_j$, then $x_j^{-k} x_i^k m \in I$ holds, for
all $k \le_p \ell$.  Here, $x_j^{\ell} \Vert m$ means $x_j^{\ell}$
divides $m$ but $x_j^{\ell+1}$ does not; $k \le_p \ell$ means that in
the base-$p$ expansions $k = \sum_i k_i p^i, \ell = \sum_i \ell_i
p^i$, we have $k_i \le \ell_i$, for all $i$; also see
\cite[\S~15.9.3]{Eisenbud--1995}.

When $\kk$ has characteristic $0$, Theorem~\ref{thm:SSSunique2}
generalizes a result of Gotzmann.  In fact, the Hilbert schemes with
$b_r > 0$ are the irreducible ones in \cite{Gotzmann--1989}.  These
include Grassmannians and the Hilbert schemes of hypersurfaces studied
in \cite{Aadlandsvik--1985}.  We extend our classification to positive
characteristic.

\begin{corollary}
  \label{cor:SSSunique}
  Let $\kk$ be an algebraically closed field and $\hp, c$ be as in
  Theorem~\ref{thm:SSSunique2}.  Then $\hilb^{\hp}(\PP^n)$ has a
  unique Borel-fixed point, where $n = c + \deg \hp$.
\end{corollary}

\begin{proof}
  Let $I$ be the saturated Borel-fixed ideal of a point on
  $\hilb^{\hp}(\PP^n)$.  So $I$ is monomial and no $g \in G(I)$ is
  divisible by $x_n$, by \cite[II, Proposition~9]{Pardue--1994}.
  Suppose $b_r > 0$.  As $x_n$ is not a zero-divisor of $\kk[x_0, x_1,
    \ldots, x_n] / I$, the result \cite[Proposition~2]{Gotzmann--1989}
  applies directly, showing that $I = \langle f_0 x_0, f_0 f_1 x_1,
  f_0 f_1 f_2 x_2, \ldots, f_0 f_1 \cdots f_{k-1} x_{k-1}, f_0 f_1
  \cdots f_k \rangle$, where $f_i \in \kk[x_i, \ldots, x_n]$ is
  homogeneous and $k \le n-2$;
  cf.\ \cite[Theorem~4.1]{Reeves--Stillman--1997}.  We may assume each
  $f_i$ is monomial, of degree $d_i$.  Then Pardue's criterion shows
  that $I$ contains the saturated lexicographic ideal
  \[
  L := \langle x_0^{d_0 +1}, x_0^{d_0} x_1^{d_1 +1}, x_0^{d_0}
  x_1^{d_1} x_2^{d_2 +1}, \ldots, x_0^{d_0} x_1^{d_1} \cdots
  x_{k-1}^{d_{k-1} +1}, x_0^{d_0} x_1^{d_1} \cdots x_k^{d_k} \rangle.
  \]
  The degrees $d_i$ determine the Hilbert polynomials of $I$ and $L$,
  so we must have $I = L = L^{\hp}_n$.

  Let $c \ge 2$ and $b_r = 0$. The cases $\hp = 1$ and $\hp =2$ follow
  by inspection.  When $r = 2$ and $b_1 > 0$, consider the ideal $J :=
  I \cap \kk[x_0, x_1, \ldots, x_{n-1}]$, which is Borel-fixed, has
  Hilbert polynomial $\nabla(\hp)$, and satisfies $\lift(J) = I$ as
  $I$ is saturated.  By the previous cases, we know that $(J :
  x_{n-1}^{\infty}) = L^{\nabla(\hp)}_{n-1} = \langle x_0, x_1,
  \ldots, x_{c-1} \rangle$, which further implies $(I :
  x_{n-1}^{\infty}) = L^{\hp -1}_n$, by
  Lemma~\ref{lem:HpolyBorel}(iii),
  Proposition~\ref{prop:lexextension}, and \cite[II,
    Proposition~9]{Pardue--1994}.  So $I_2 \subset (L^{\hp -1}_n)_2$
  is a codimension $1$ vector subspace, the latter being spanned by
  $x_0^2, x_0 x_1, \ldots, x_{c-1}^2, x_{c-1} x_c, \ldots, x_{c-1}
  x_n$.  Now Pardue's criterion implies that we can only obtain $I_2$
  by removing $x_{c-1} x_n$, i.e.\ $I$ is the lex-expansion of $L^{\hp
    -1}_n$, as $I$ is generated in degree $r = 2$.

  Now let $c = 1$ and $b_r = 0$. If $b_1 = b_r$, then $I$ is
  lexicographic, as it is monomial and saturated.  If $b_1 > b_r$,
  then the argument is entirely analogous to the previous paragraph.
\end{proof}

\begin{proof}[Proof of Theorem~\ref{thm:SSSunique2intro}]
  Theorem~\ref{thm:SSSunique2} and Corollary~\ref{cor:SSSunique} prove
  the claim.
\end{proof}

Geometrically, the Hilbert schemes corresponding to
Theorem~\ref{thm:SSSunique2} are well-behaved.

\begin{lemma}
  \label{lem:Hilbertirred}
  Let $\kk$ be an infinite field.  If the lexicographic ideal is the
  unique saturated Borel-fixed ideal with Hilbert polynomial $\hp$ and
  codimension $c$, then $\hilb^{\hp}(\PP^n)$ is nonsingular and
  irreducible, where $n = c + \deg \hp$.
\end{lemma}

\begin{proof}
  Every component and intersection of components of
  $\hilb^{\hp}(\PP^n)$ contains a point $\left[ X_I \right]$ defined
  by a saturated Borel-fixed ideal $I$; the Remarks in
  \cite[\S~2]{Reeves--1995} hold over infinite fields, by
  \cite[Proposition~1]{Bayer--Stillman--1987b} and
  \cite[Theorem~15.17]{Eisenbud--1995}.  Lexicographic points are
  nonsingular by \cite[Theorem~1.4]{Reeves--Stillman--1997}, so
  $\left[ X_{L^{\hp}_n} \right]$ cannot lie on an intersection of
  components.  Thus, $\hilb^{\hp}(\PP^n)$ has a unique, generically
  nonsingular, irreducible component.

  Suppose $\hilb^{\hp}(\PP^n)$ has a singular point, given by $I
  \subset \kk[x_0, x_1, \dotsc, x_{n}]$.  For $\gamma \in
  \operatorname{GL}_{n+1}(\kk)$, the point $\left[ X_{\gamma \cdot I}
    \right]$ is also singular, and for generic $\gamma \in
  \operatorname{GL}_{n+1}(\kk)$, the initial ideal of $\gamma \cdot I$
  with respect to any monomial ordering is Borel-fixed, by Galligo's
  Theorem \cite{Galligo--1974, Bayer--Stillman--1987b}.  Thus, a
  one-parameter family of singular points degenerating to the
  lexicographic ideal exists.  By upper semicontinuity of cohomology
  of the normal sheaf, the lexicographic ideal is singular, a
  contradiction; see \cite[III, Theorem~12.8]{Hartshorne--1977},
  \cite[Theorem~1.1(b)]{Hartshorne--2010}.  Hence,
  $\hilb^{\hp}(\PP^n)$ is nonsingular and irreducible.
\end{proof}

Thus, our classification provides irreducible, nonsingular Hilbert
schemes over algebraically closed or characteristic $0$ fields.
Moreover, these Hilbert schemes are rational, by \cite[Theorem
  C]{Lella--Roggero--2011}.

We now wish to make a quantitative statement about the prevalence of
this behaviour.  To do so, we need a probability measure $\prob \colon
2^{\hilbtree} \to [0,1]$, which is determined by a normalized
nonnegative function $\pr \colon \hilbtree \to \RR$, as in
\cite[Examples~2.8--2.9]{Billingsley--1995}.  That is, for every
subset $\mathscr{N} \subseteq \hilbtree$, we have $\prob(\mathscr{N})
:= \sum_{H \in \mathscr{N}} \pr(H)$.  However, there is no known
canonical distribution on $\hilbtree$.  A natural choice on
$\hilbtree_c$ is to mimic uniform distribution by making all vertices
at a fixed height equally likely; given a mass function $f_c \colon
\NN \to [0,1]$, let $\pr_c (H) = 2^{-k} f_c( k )$, for all $H \in
\hilbtree_c$ at height $k$.  Distributions on $\hilbtree$ are then
specified via functions $f_c \colon \NN \to [0,1]$, for all
$\hilbtree_c$, and a mass function $f \colon \NN \setminus \{ 0 \} \to
[0,1]$, by setting $\pr (H) = 2^{-k} f(c) f_c(k)$ (the probability of
$H$), for all $H \in \hilbtree_c$ at height $k$.  Using $f_c(k) :=
2^{-k-1}$ and $f(c) := 2^{-c}$ is sufficient for us, although other
basic examples (geometric, Poisson, etc.) work similarly.

\begin{theorem}
  \label{thm:Hilbertprob}
  Let $\kk$ be an algebraically closed or characteristic $0$ field and
  endow $\hilbtree$ with the structure of a probability space as in
  the preceding paragraph.  The probability that a random Hilbert
  scheme is irreducible and nonsingular is greater than $0.5$.
\end{theorem}

\begin{proof}
  Let $\mathscr{N}$ be the set of nonsingular and irreducible Hilbert
  schemes.  We compute
  \begin{align*}
    \prob(\mathcal{N}) = \sum_{H \in \mathscr{N}} \pr \left( H \right)
    &\ge \sum_{c > 0} \left( f(c) f_c(0) + \sum_{k \ge 1}
    \frac{2^k}{2} \frac{f(c) f_c(k)}{2^k} \right) \\
    &= \sum_{c > 0} \left( \frac{f(c) f_c(0)}{2} + \sum_{k \in \NN}
    \frac{f(c) f_c(k)}{2} \right) \\
    &= \sum_{c > 0} \left( \frac{f(c) f_c(0)}{2} + \frac{f(c)}{2}
    \right) = \sum_{c > 0} \left( \frac{f(c) f_c(0)}{2} \right) +
    \frac{1}{2},
  \end{align*}
  because for all trees $\hilbtree_c$ and heights $k \ge 1$, there are
  at least $2^{k-1}$ vertices corresponding to nonsingular and
  irreducible Hilbert schemes, by Theorem~\ref{thm:SSSunique2} and
  Lemma~\ref{lem:Hilbertirred}.  The computation is similar in a fixed
  codimension.  Hence, the probability that a Hilbert scheme is
  nonsingular and irreducible is greater than $0.5$.
\end{proof}

\begin{proof}[Proof of Theorem~\ref{thm:Hilbertprobintro}]
  Theorem~\ref{thm:Hilbertprob} proves the claim.
\end{proof}

\bibliography{ubiquity}{}
\bibliographystyle{amsalpha}

\end{document}